\documentclass{article}
\usepackage{latexsym}
\usepackage{amssymb}
\usepackage{amsfonts}
\usepackage{amsmath}
\usepackage{graphicx}
\usepackage{amsfonts}
\usepackage{epstopdf}
\usepackage{color}
\usepackage{enumerate}
\usepackage[latin1]{inputenc}
\usepackage{mathtools}
\usepackage{amsthm}
\usepackage[latin1]{inputenc}
\usepackage{mathrsfs}

\usepackage[dvipsnames]{xcolor}

\newtheorem{teo}{Theorem}[section]
\newtheorem{lema}[teo]{Lemma}

\newtheorem{prop}[teo]{Proposition}

\newcommand{\C}{{\mathbb C}}
\newcommand{\E}{{\mathbf E}}

\newcommand{\N}{{\mathbb N}}

\newcommand{\R}{{\mathbb R}}

\newcommand{\sgn}{\operatorname{sgn}}

\def\1{{\rm l}\hskip -0.21truecm 1}

\begin{document}
\title{The complex Brownian motion as a strong limit of processes constructed from a Poisson process}
\date{}
\author{Xavier Bardina\footnote{X. Bardina is supported by the grant MTM2012-33937 from SEIDI,  Ministerio de Economia y
Competividad.}, Giulia Binotto$^\dagger$
 and Carles Rovira\footnote{ G. Binotto and C. Rovira are supported by the grant MTM2012-31192 from SEIDI,  Ministerio de Economia y Competividad.}}

\maketitle

$^*${\rm Departament de Matem\`atiques, Facultat de Ci\`encies,
Edifici C, Universitat Aut\`onoma de Barcelona, 08193 Bellaterra}.
{\tt Xavier.Bardina@uab.cat}
\newline
$\mbox{ }$\hspace{0.1cm} $^\dagger${\rm Facultat de Matem\`atiques,
Universitat de Barcelona, Gran Via 585, 08007 Barcelona}. {\tt gbinotto@ub.edu},  {\tt carles.rovira@ub.edu}

\begin{abstract}
We construct a family of processes, from a single
Poisson process, that converges in law to a complex Brownian motion.
Moreover, we find realizations of these processes that converge
almost surely to the complex Brownian motion, uniformly on the unit
time interval. Finally the rate of convergence is derived.

\end{abstract}

\section{Introduction}

 Kac \cite{K} in 1956 to
obtain a solution from a Poisson  of the telegraph equation
\begin{equation}\label{tele} \frac1v \frac{\partial^2 F}{\partial
t^2}=v\frac{\partial^2 F}{\partial x^2}-\frac{2a}{v}\frac{\partial
F}{\partial t},\end{equation} with $a,v>0$, introduced the processes
$$x(t)=v\int_0^t(-1)^{N_a(r)}dr,$$
where $N_a=\{N_a(t),\,t\geq0\}$ is a Poisson process of intensity
$a$. He noticed that if in  equation (\ref{tele}) the parameters $a$ and $v$ tend to
infinity with $\frac{2a}{v^2}$ constant and equal to $\frac1{D}$,
then the equation converges to the heat equation:
\begin{equation}\label{calor} \frac1{D}\frac{\partial F}{\partial t}=\frac{\partial^2
F}{\partial x^2}.\end{equation}
 Let $x_\varepsilon(t)$ be the processes considered by Kac with
$a=\frac1{\varepsilon^2}$, $v=\frac1{\varepsilon}$. These values
satisfy that $\frac{2a}{v^2}$ is constant and $D=\frac12$ and we get
in (\ref{calor}) an equation whose solution is a standard Brownian
motion.

Stroock \cite{S} proved in 1982 that the processes $x_{\varepsilon}$
converge in law to a standard Brownian motion. That is, if we
consider $(P^{\varepsilon})$ the image law of the process
$x_{\varepsilon}$ in the Banach space $\mathcal C([0,T])$ of
continuous functions on $[0,T]$, then $(P^{\varepsilon})$ converges
weakly, when $\varepsilon$ tends to zero, towards the Wiener
measure.
 Doing a change of variables, these processes
can be written as
$$x_{\varepsilon}=\left\{x_{\varepsilon}(t):=\varepsilon\int_0^{\frac{t}{\varepsilon^2}}(-1)^{N(u)}du,\,t\in[0,T]\right\},$$
where $\{N(t),\,t\geq0\}$ is a standard Poisson process.

In the mathematical  literature we find  generalizations with regard to the Stroock
result which can be channeled in three directions:
\begin{enumerate}[(i)]
 \item modifying the processes $x_{\varepsilon}$ in order to obtain
approximations of other Gaussian processes, \item proving
convergence in a stronger sense that the convergence in law in the
space of continuous functions,  \item weakening the conditions of
the approximating processes.
\end{enumerate}

In direction (i), a first generalization is  also made by Stroock
\cite{S} who modified the processes $x_{\varepsilon}$ to obtain
approximations of stochastic differential equations.  There are also
generalizations, among others, to the fractional Brownian motion
(fBm) \cite{LD}, to a general class of Gaussian processes (that
includes fBm) \cite{DJ}, to a fractional stochastic differential
equation \cite{BNRT},  to the stochastic heat equation driven by
Gaussian white noise \cite{BJQ} or to  the Stratonovich heat
equation \cite{DJQ}.

On the other hand, there is some literature  where the authors  present
realizations of the processes that converge almost surely,  uniformly on the unit time interval. These processes are usually called as uniform
transport processes. Since the approximations always start increasing, a modification
of the processes  as
\begin{equation*}
\tilde
x_{\varepsilon}(t)=\varepsilon(-1)^{A}\int_0^{\frac{t}{\varepsilon^2}}(-1)^{N(u)}du,\end{equation*}
has to be considered where $A\sim
\textrm{Bernoulli}\left(\frac12\right)$  independent of the Poisson
process $N$.

Griego, Heath and Ruiz-Moncayo  \cite{art G-H-RM}  showed that these
processes converge strongly and uniformly on bounded time intervals
to Brownian motion. In  \cite{art G-G2} Gorostiza and Griego
extended the result to diffusions. Again Gorostiza and Griego
\cite{art G-G} and Cs\"{o}rg\H{o} and Horv\'ath \cite{CH} obtained a
rate of convergence. More precisely, in \cite{art G-G} it is proved
that there exist versions of the transport processes $\tilde
x_{\varepsilon}$ on the same probability space as a given Brownian
motion $(W(t))_{t\geq0}$ such that, for each $q > 0$,
$$P\left(\sup_{a\leq t\leq b}|W(t)-\tilde x_{\varepsilon}(t)|> {C}\varepsilon\left(\log
\frac1{\varepsilon^2}\right)^{\frac52}
\right)=o\left({{\varepsilon}^{2q}}\right),$$ as $\varepsilon\to 0$
and where $C$ is a positive constant depending on $a$, $b$ and $q$.
Garz\'on, Gorostiza and Le\'on \cite{GGL} defined a sequence of
processes that converges strongly to fractional Brownian motion
uniformly on bounded intervals, for any Hurst parameter $H\in(0,1)$
and computed the rate of convergence. In \cite{GGL2} and \cite{GGL3}
the same authors deal with subfractional Brownian motion and
fractional stochastic differential equations.

Since
$$(-1)^{N(u)}=e^{i \pi N(u)}=\cos(\pi N(u)),$$
 the question that if the convergence is also true with
other angles appears.
Bardina \cite{B}  showed that if we consider
\begin{equation}\label{dedede}\bar x^\theta_{\varepsilon}(t)={\varepsilon}\int_0^{\frac{2t}{\varepsilon^2}}e^{i \theta N_s}ds
\end{equation}
where $\theta\neq0,\pi,$ the laws of the processes converge weakly
towards the law of a complex Brownian motion, i.e., the laws of the real and imaginary parts
$$\bar z^\theta_{\varepsilon}(t)={\varepsilon}\int_0^{\frac{2t}{\varepsilon^2}}\cos(\theta N_s)ds$$
and
$$\bar y^\theta_{\varepsilon}(t)={\varepsilon}\int_0^{\frac{2t}{\varepsilon^2}}\sin(\theta N_s)ds,$$
converge weakly towards the law of two independent Brownian motions.
The approximating processes are functionally dependent because we
use a single Poisson process but,
 in the limit, we obtain two independent processes.
Later, in \cite{BR1}  it is shown that for different angles $\theta_i$ the corresponding processes
converge in law towards independent Brownian motions despite
using only one Poisson process. Finally, in  \cite{BR2}, we prove that we can use a L\'evy process instead
of a  Poisson process in the definition of the sequence of approximations.

In this paper we present an extension of the Kac-Stroock result in
the directions (ii) and (iii). Our aim is to define a modification
of the processes $\bar x_{\varepsilon}$ used by Stroock,  similar to
(\ref{dedede}) proposed in \cite{B}. These complex  processes, that
we will denote by $x_\varepsilon^\theta=z_\varepsilon^\theta+ i
y_\varepsilon^\theta$, will depend on a parameter $\theta \in
(0,\pi) \cup (\pi,2\pi)$ and will be defined from an unique standard
Poisson process  and a sequence of independent random variables with
common distribution Bernoulli($\frac12$). We will check that  if we
consider $\theta_1,\theta_2,\dots,\theta_m$ such that for all $i\neq
j$, $1\leq i,j\leq m$, $\theta_i,\theta_j\in(0,\pi)\cup(\pi,2\pi)$,
$\theta_i+\theta_j\neq 2\pi$ and $\theta_i\neq\theta_j$,  the law of
the processes
$$(z^{\theta_1}_\varepsilon,\dots,z^{\theta_m}_\varepsilon,y^{\theta_1}_\varepsilon,\dots,y^{\theta_m}_\varepsilon)$$
converges weakly in the space of the continuous functions towards the
joint law of $2m$ independent Brownian motions. Moreover, we also prove  that there exist realizations of
$x_\varepsilon^\theta$ that converge almost surely to a complex Brownian
motion and we are able to obtain the  rate of convergence that does not depend on $\theta$.
As a consequence, simulating a sequence of independent random variables with common distribution exponential(1)  and a sequence of independent random variables with common distribution Bernoulli($\frac12$), we can get sequences of almost sure approximations of $d$ independent Brownian motions for any $d$.

%Notice that we deal with approximations indexed by $\varepsilon$ and we will consider what happens when $\varepsilon$ goes to zero. Comparing with Stroock results there is the equivalence $n=\frac{1}{\varepsilon^2}$. On the other hand,
For simplicity's sake, we only consider
$\theta\in(0,\pi)\cup(\pi,2\pi)$ for which it does not exist any
$m\in\N$ such that $\cos(m\theta)=0$ or $\sin(m\theta)=0$.

As usual, the weak convergence is proved  using  tightness and the identification of the law of all possible weak limits
(see, for instance \cite{B}, \cite{BR1}). The almost sure converge is inspired in \cite{art G-H-RM}  while the computation of the rate of convergence follows the  method given in
\cite{art G-G} and \cite{art G-G2}.

The paper is organized in the following way. Section 2 is devoted to define the processes  and to give the main results. In Section 3
we prove the weak convergence theorems. In Section
4  we prove the strong convergence theorem. The proof of the rate of convergence is  given in Section 5. Finally, there is an Appendix with some technical results.

Throughout  the paper $K, C$ will denote any positive constant, not
depending on $\varepsilon$, which may change from one expression to
another.

\section{Main results}

Let $\{M_t,t\geq0\}$ be a Poisson process of parameter 2. We define $\{N_t,t\geq0\}$ and $\{N'_t,t\geq0\}$ two other counter processes
  that, at each jump of $M$, each of them jumps or  does not jump with probability $\frac12$, independently of the jumps of the other process and of its past.

In Proposition \ref{2poisson} (see Appendix) we prove that $N$ and $N'$ are Poisson processes of parameter 1 with independent increments on disjoint   intervals.

Consider now, for $\theta\in(0,\pi)\cap(\pi,2\pi)$, the following
processes:
    \begin{equation}\label{process}
      \bigg\{ z_\varepsilon^\theta(t) = (-1)^G \,\varepsilon \,\int_0^{\frac{2t}{\varepsilon^2}} (-1)^{N'_r} \,e^{i\theta N_r}\,dr, \quad t\in[0,T]
      \bigg\},
    \end{equation}
  where $N$ and $N'$ are the processes defined above and $G$ is a random variable, independent of $N$ and $N'$, with Bernoulli distribution of parameter $\frac12$.
We can write the process $z_\varepsilon^\theta(t)$ as $z_\varepsilon^\theta(t)=x_\varepsilon^\theta(t)+iy_\varepsilon^\theta(t)$, where
    $$ x_\varepsilon^\theta(t) := \varepsilon\int_0^{\frac{2t}{\varepsilon^2}}(-1)^{N'_s+G} \cos(\theta N_s) \,ds , \quad y_\varepsilon^\theta(t) := \varepsilon\int_0^{\frac{2t}{\varepsilon^2}}(-1)^{N'_s+G} \sin(\theta N_s) \,ds $$
are the real part and the imaginary part, respectively.
In Figure \ref{fig1} we can see a simulation of the trajectories of
these processes for different values of $\theta$.
\begin{figure}\label{fig1}
  \centering
    \includegraphics[width=.5\textwidth]{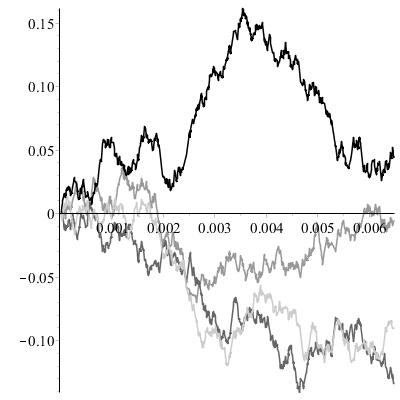}
  \caption{Simulation of the trajectories of the processes $x_{\varepsilon}^{\theta}$ and $y_{\varepsilon}^{\theta}$ for the values of the parameters
$\varepsilon=\frac{1}{200}$, $\theta=2$ and $\theta=7$.}
  \label{fig1}
\end{figure}

\bigskip

Our first result gives us the weak convergence to a complex Brownian motion.
\begin{teo}\label{teofeble}
Let $P_\varepsilon^\theta$ be the image law of
$z_\varepsilon^\theta$ in the Banach space $\mathcal{C}([0,T],\C)$
of continuous function on $[0,T]$. Then $P_\varepsilon^\theta$
converges weakly when $\varepsilon$ tends to zero to the law
$P^\theta$
  on $\mathcal{C}([0,T],\C)$ of a complex Brownian motion.
\end{teo}

\begin{proof} See Section \ref{feble}. \end{proof}

We can also get the following extension of Theorem
\ref{teofeble}, that is the equivalent of the result obtained in
\cite{BR1} for our processes.
\begin{teo}\label{general} Consider $\theta_1,\theta_2,\dots,\theta_m$
such that for all $i\neq j$, $1\leq i,j\leq m$,
$\theta_i,\theta_j\in(0,\pi)\cup(\pi,2\pi)$, $\theta_i+\theta_j\neq
2\pi$ and $\theta_i\neq\theta_j$. Then the laws of the processes
$$(x^{\theta_1}_{\varepsilon},\dots,x^{\theta_m}_{\varepsilon},y^{\theta_1}_{\varepsilon},\dots,y^{\theta_m}_{\varepsilon})$$
converge weakly, in the space of the continuous functions, towards
the joint law of $2m$ independent Brownian motions.
\end{teo}

\begin{proof} See the end of Section \ref{feble}. \end{proof}

\bigskip

Our next result gives the strong convergence of realizations of our processes $\{z_\varepsilon^\theta(t);\,t\in[0,1]\}$
and states as follows:

\begin{teo}\label{resultat}
There exists realizations of the process $z_{\varepsilon}^{\theta}$
on the same probability space as a complex Brownian motion
$\{z(t),t\geq0\}$ such that
  $$ \lim_{\varepsilon\rightarrow 0} \max_{0\leq t\leq1} |z_\varepsilon^\theta(t)-z(t)|=0 \quad a.s. $$
\end{teo}

\begin{proof} See Section \ref{cap_realp}. \end{proof}

\bigskip

Notice that combining the results of Theorem \ref{general}  and Theorem \ref{resultat} we get that  from our two Poisson processes $N$ and $N'$ and a random variable $G$ with Bernoulli law, we are able to construct approximations to $d$ standard independent Brownian motions  for $d$ as large as we want.

In our last result we give the rate of convergence of these processes.
 \begin{teo}\label{thm_rate}
      For all $q>0$,
        $$ P\left( \max_{0\leq t\leq1} |z_\varepsilon^\theta(t)-z(t)| > \alpha^*\,\varepsilon^{\frac12}\left(\log{\frac1\varepsilon}\right)^\frac52 \right) = o(\varepsilon^q),  \qquad \mbox{as} \quad \varepsilon\rightarrow0 $$
      where $\alpha^*$ is a positive constant depending on $q$.
    \end{teo}

\begin{proof} See Section \ref{rates}. \end{proof}

\section{Proof of  weak convergence}\label{feble}

In order to prove Theorem \ref{teofeble} we have to check that the family
$P_\varepsilon^\theta$
  is tight and that the law of all possible limits of $P_\varepsilon^\theta$ is the law of a complex Brownian motion. Following the same method that in \cite{B}, the proof  is based on the following  lemma:

\begin{lema}\label{lematecnic}For any $0\leq x_1\leq x_2$
    \begin{eqnarray*}
      \E\big[(-1)^{N'_{x_2}-N'_{x_1}}  e^{i\theta(N_{x_2}-N_{x_1})}\big] = e^{-2(x_2-x_1)}.
    \end{eqnarray*}
\end{lema}

\begin{proof}
From the definition of $N$ and $N'$ it follows that
    \begin{eqnarray*}
      && \E\big[(-1)^{N'_{x_2}-N'_{x_1}} e^{i\theta(N_{x_2}-N_{x_1})}\big]\\
      & =& \sum_{n=0}^\infty\sum_{m=0}^\infty (-1)^n e^{i\theta m} P\big(N'_{x_2}-N'_{x_1}=n,N_{x_2}-N_{x_1}=m\big) \\
      &=&\!\!\!\!\! \sum_{n=0}^\infty\sum_{m=0}^\infty (-1)^n e^{i\theta m}\!\!\!\!\! \sum_{k=n\vee m}^\infty \!\!\! P\big(N'_{x_2}-N'_{x_1}=n,N_{x_2}-N_{x_1}=m \big| M_{x_2}-M_{x_1}=k\big) \\
      && \times\, P(M_{x_2}-M_{x_1}=k) \\
%      && \hspace{15mm} = \sum_{n=0}^\infty\sum_{m=0}^\infty (-1)^n e^{i\theta m} \sum_{k=n\vee m}^\infty P\big(N'_{x_2}-N'_{x_1}=n \big| M_{x_2}-M_{x_1}=k\big) \times \\
%      && \hspace{6cm} \times\, P\big(N_{x_2}-N_{x_1}=m \big| M_{x_2}-M_{x_1}=k\big)P(M_{x_2}-M_{x_1}=k) \\
      &=&\sum_{n=0}^\infty\sum_{m=0}^\infty (-1)^n e^{i\theta m} \sum_{k=n\vee m}^\infty {k\choose n}{k\choose m}\frac{1}{2^k}\frac{1}{2^k}\frac{[2(x_2-x_1)]^ke^{-2(x_2-x_1)}}{k!} \\
      &=& e^{-2(x_2-x_1)} \sum_{k=0}^\infty \left(\frac{x_2-x_1}{2}\right)^k\frac{1}{k!} \sum_{n=0}^k{k\choose n}(-1)^n \sum_{m=0}^k{k\choose m}e^{i\theta
      m}.
    \end{eqnarray*}
 Notice that $\sum_{n=0}^k{k\choose n}(-1)^n=0$ when $k\neq0$, therefore the above expression is different from zero only when $k=0$ and, as a consequence, when $n=0$ and $m=0$. Hence, as the series is absolutely convergent,
      $$ \E\big[(-1)^{N'_{x_2}-N'_{x_1}}  e^{i\theta(N_{x_2}-N_{x_1})}\big] = e^{-2(x_2-x_1)}, $$
as we wanted to prove.
    \end{proof}

Using Lemma \ref{lematecnic}, we can also get a version of  Lemma 3.2 in
    \cite{B} well adapted to our processes.
 \begin{lema}\label{lemma2}
      Consider $\{\mathcal{F}_t^{\varepsilon,\theta}\}$ the natural filtration of the processes $z_\varepsilon^\theta$. Then, for any $s<t$ and for any real $\{\mathcal{F}_s^{\varepsilon,\theta}\}$-measurable and bounded random variable $Y$, we have that,  for any $\theta\in(0,\pi)\cup(\pi,2\pi)$,
        \begin{enumerate}[a)]
          \item $\displaystyle\varepsilon^2\int_{\frac{2s}{\varepsilon^2}}^{\frac{2t}{\varepsilon^2}}\int_{\frac{2s}{\varepsilon^2}}^{x_2} \!\!\E\left[(-1)^{N'_{x_2}-N'_{x_1}}  e^{i\theta(N_{x_2}-N_{x_1})}\right]\! dx_1dx_2 \!=\! (t-s)+\frac{\varepsilon^2}{4}(e^{-\frac{4}{\varepsilon^2}(t-s)}-1)$
          \item $\displaystyle\lim_{\varepsilon\rightarrow0} \Big| \varepsilon^2\int_{\frac{2s}{\varepsilon^2}}^{\frac{2t}{\varepsilon^2}}\int_{\frac{2s}{\varepsilon^2}}^{x_2}\E\left[(-1)^{N'_{x_2}+N'_{x_1}} e^{i\theta(N_{x_2}+N_{x_1})}Y\right]\,dx_1dx_2 \Big|=0$.
        \end{enumerate}
    \end{lema}

\begin{proof}
Follow the same ideas that in  \cite{B} using Lemma \ref{lematecnic}.
\end{proof}

{\it Proof of  Theorem \ref{teofeble}}.  We will give only the skeleton  of the proof.

We need to prove that the laws corresponding to
$\{x_\varepsilon^\theta(t), t\geq0\}$ and
$\{y_\varepsilon^\theta(t), t\geq0\}$ are tight, using
Billingsley criterion and that our processes are null in the origin,
it is sufficient to check that there exists a constant $K$ such that
for any $s<t$
       \begin{eqnarray*}
&& \sup_\varepsilon\left[\E\left(\varepsilon\int_{\frac{2s}{\varepsilon^2}}^{\frac{2t}{\varepsilon^2}}(-1)^{N'_x+G}\cos(\theta N_x)\,dx\right)^4  \right.  \\
&& \qquad\quad    \left.    +\E\left(\varepsilon\int_{\frac{2s}{\varepsilon^2}}^{\frac{2t}{\varepsilon^2}}(-1)^{N'_x+G}\sin(\theta N_x)\,dx\right)^4\right]\leq K(t-s)^2. \end{eqnarray*}
Following the proof of lemma 2.1 in \cite{B}  and applying  Lemma \ref{lematecnic} we obtain that
\begin{eqnarray*}
  && \E\left(\varepsilon\int_{\frac{2s}{\varepsilon^2}}^{\frac{2t}{\varepsilon^2}}(-1)^{N'_x+G}\cos(\theta N_x)\,dx\right)^4+\E\left(\varepsilon\int_{\frac{2s}{\varepsilon^2}}^{\frac{2t}{\varepsilon^2}}(-1)^{N'_x+G}\sin(\theta N_x)\,dx\right)^4 \\
   && \leq 12\left(\varepsilon^2\int_{\left[\frac{2s}{\varepsilon^2},\frac{2t}{\varepsilon^2}\right]^2}\1_{\{x_1\leq x_2\}} \Big| \,\E\big[(-1)^{N'_{x_2}-N'_{x_1}} e^{i\theta(N_{x_2}-N_{x_1})}\big]\Big|\,dx_1\,dx_2\right)^2 \\
  && \quad + 48\,\varepsilon^4\,\int_{\mathcal{I}} \big| \E\big[(-1)^{N'_{x_4}-N'_{x_3}} e^{i\theta(N_{x_4}-N_{x_3})}\big] \big|  \big| \E\big[e^{2i\theta(N_{x_3}-N_{x_2})}\big] \big| \,dx_1\dots dx_4 \\
&&  \leq
12\left(\varepsilon^2\int_{\frac{2s}{\varepsilon^2}}^{\frac{2t}{\varepsilon^2}} \int_{\frac{2s}{\varepsilon^2}}^{x_2} e^{-2(x_2-x_1)} \,dx_1\,dx_2\right)^2   \\ & &\qquad \qquad +  48\,\varepsilon^4\,\int_{\mathcal{I}} \,e^{-2(x_4-x_3)} \,e^{-(x_3-x_2)(1-\cos(2\theta))} \,dx_1\dots dx_4
\\
  &&  \leq 12(t-s)^2 + \frac{48(t-s)^2}{1-\cos(2\theta)}
\end{eqnarray*}
where
$\mathcal{I}:=\{(x_1,x_2,x_3,x_4)\in\left[\frac{2s}{\varepsilon^2},\frac{2t}{\varepsilon^2}\right]^4:\,x_1\leq
x_2\leq x_3\leq x_4\}$.
So the family of laws is tight.

 Now we need to identify all the
possible limit laws. Consider a subsequence, which we
will also denote by $\{P_{\varepsilon}^{\theta}\}$, weakly
convergent to some probability $P^{\theta}$. We want to prove that
 the canonical process $Z^{\theta}=\{z^{\theta}(t),t\geq0\}$ is a complex Brownian motion under $P^\theta$, that is the real part $X$ and the imaginary part $Y$ of this process are two independent Brownian motions.
  Using Paul L\'evy's theorem it is sufficient to prove that, under $P^\theta$, $X$ and $Y$ are both martingale with respect to the natural filtration $\{\mathcal{F}_t\}$ with quadratic variation $<Re[Z^{\theta}],Re[Z^{\theta}]>_t=t$, $<Im[Z^{\theta}],Im[Z^{\theta}]>_t=t$ and null covariation.

  To prove the matingale property with respect to the natural filtration $\{\mathcal{F}_t\}$,
following the Section 3.1 in \cite{B},
  it is enough to see that  for any $s_1\leq s_2\leq \cdots\leq s_n\leq s<t$ and  for any bounded continuous function $\varphi:\C^n\rightarrow\R$
    $$ \left| \E\left(\varphi\big(z_\varepsilon^\theta(s_1),\dots,z_\varepsilon^\theta(s_n)\big)\varepsilon\int_{\frac{2s}{\varepsilon^2}}^{\frac{2t}{\varepsilon^2}}(-1)^{G+N'_x}\,e^{i\theta N_x}\,dx\right) \right| $$
converges to zero as $\varepsilon$ tends to zero. But,
    \begin{eqnarray*}
      && \left| \E\left(\varphi\big(z_\varepsilon^\theta(s_1),\dots,z_\varepsilon^\theta(s_n)\big)\varepsilon\int_{\frac{2s}{\varepsilon^2}}^{\frac{2t}{\varepsilon^2}}(-1)^{G+N'_x}\,e^{i\theta N_x}\,dx\right) \right| \\
      &=&\left| \E\left(\varphi\big(z_\varepsilon^\theta(s_1),\dots,z_\varepsilon^\theta(s_n)\big)(-1)^{G+N'_{2s/\varepsilon^2}}\,e^{i\theta
      N_{2s/\varepsilon^2}}\right)\right.
      \\&&\times\left.\varepsilon\int_{\frac{2s}{\varepsilon^2}}^{\frac{2t}{\varepsilon^2}}\E\left((-1)^{N'_x-N'_{2s/\varepsilon^2}}\,e^{i\theta(N_x-N_{2s/\varepsilon^2})}\right)\,dx \right| \\
      &\leq&K\varepsilon\int_{\frac{2s}{\varepsilon^2}}^{\frac{2t}{\varepsilon^2}} e^{-2(x-\frac{2s}{\varepsilon^2})} \,dx
      =\frac{K\varepsilon}{2} \big(1-e^{-\frac{4}{\varepsilon^2}(t-s)}\big),
    \end{eqnarray*}
that converges to zero as $\varepsilon$ tends to zero. Therefore, martingale property is proved.

To prove that $<Re[Z^{\theta}],Re[Z^{\theta}]>_t=t$ and $<Im[Z^{\theta}],Im[Z^{\theta}]>_t=t$ we will check that for any $s_1\leq\dots\leq s_n\leq s<t$ and for any bounded continuous function $\varphi:\C^n\rightarrow\R$,
    $$ \E\Big[ \varphi\big(z_\varepsilon^\theta(s_1),\dots,z_\varepsilon^\theta(s_n)\big) \big((x_\varepsilon^\theta(t)-x_\varepsilon^\theta(s))^2-(t-s)\big) \Big] $$
  and
    $$ \E\Big[ \varphi\big(z_\varepsilon^\theta(s_1),\dots,z_\varepsilon^\theta(s_n)\big) \big((y_\varepsilon^\theta(t)-y_\varepsilon^\theta(s))^2-(t-s)\big) \Big] $$
  converge to zero as $\varepsilon$ tends to zero.   Notice that, in our case,
    \begin{eqnarray*}
      && \E\Big[ \varphi\big(z_\varepsilon^\theta(s_1),\dots,z_\varepsilon^\theta(s_n)\big) \big(x_\varepsilon^\theta(t)-x_\varepsilon^\theta(s)\big)^2 \Big] \\
      &=&\E\left[ \varphi\big(z_\varepsilon^\theta(s_1),\dots,z_\varepsilon^\theta(s_n)\big) \left(\varepsilon\int_{\frac{2s}{\varepsilon^2}}^{\frac{2t}{\varepsilon^2}}(-1)^{G+N'_x}\cos(\theta N_x)\,dx\right)^2 \right] \\
      &=&2\varepsilon^2 \int_{\frac{2s}{\varepsilon^2}}^{\frac{2t}{\varepsilon^2}}\int_{\frac{2s}{\varepsilon^2}}^{x_2} \E\left[ \varphi\big(z_\varepsilon^\theta(s_1),\dots,z_\varepsilon^\theta(s_n)\big) (-1)^{N'_{x_2}+N'_{x_1}}  \right.
\\ & &\quad\qquad \times
\left. \cos(\theta N_{x_1})\cos(\theta N_{x_2}) \right] \,dx_1dx_2. \\
    \end{eqnarray*}
The proof of the real part  follows again the structure of Section 3.2 in \cite{B}
using Lemma \ref{lemma2}. The  imaginary part can be done similarly.

Finally we have to prove that $<Re[Z^{\theta}],Im[Z^{\theta}]>_t=0$.
It is sufficient to show that for any $s_1\leq\dots\leq
s_n\leq s<t$ and for any bounded continuous function
$\varphi:\C^n\rightarrow\R$,
    $$ \E\big[ \varphi\big(z_\varepsilon^\theta(s_1),\dots,z_\varepsilon^\theta(s_n)\big) \big(x_\varepsilon^\theta(t)-x_\varepsilon^\theta(s)\big) \big(y_\varepsilon^\theta(t)-y_\varepsilon^\theta(s)\big) \big] $$
  converges to zero as $\varepsilon$ tends to zero. But we obtain
  this convergence using similar calculations and statement  b) of Lemma \ref{lemma2}.

\hfill$\square$

\bigskip

{\it Proof of  Theorem \ref{general}}.  Taking into account the proof of Theorem \ref{teofeble} it remains
only to check that for $i\neq j$, and $\theta_i,\theta_j$ in the
conditions of  Theorem \ref{general}, \linebreak {\small
$<Re[Z^{\theta_i}],Re[Z^{\theta_j}]>_t=0$,
$<Im[Z^{\theta_i}],Im[Z^{\theta_j}]>_t=0$ } and {\small
$<Im[Z^{\theta_i}],Re[Z^{\theta_j}]>_t=0$. }
But it can be proved following the proof of Theorem 2 in
\cite{BR1} and taking into account our Lemma \ref{lematecnic}.

\hfill$\square$

\section{Proof of strong convergence }\label{cap_realp}%$\theta\in(0,\pi)\cup(\pi,2\pi)$}

In this section, we will prove the strong convergence when
$\varepsilon$ tends to zero of the processes
$\{z_\varepsilon^\theta(t);\,t\in[0,1]\}$ defined in Section 2. %We will follow the methodology used in \cite{art G-H-RM}.

\bigskip

{\it Proof of Theorem \ref{resultat}}. We will study the strong
convergence of $x_\varepsilon^\theta$,
  \begin{equation}
    x_\varepsilon^\theta(t)= \frac{2}{\varepsilon}(-1)^G\int_0^t(-1)^{N'_{\frac{2r}{\varepsilon^2}}}\cos{(\theta N_{\frac{2r}{\varepsilon^2}})}\,dr,  \label{processreal}
  \end{equation}
the real part of the processes
$z_\varepsilon^\theta$,  to a standard Brownian motion
$\{x_t;\,t\in[0,1]\}$ when $\varepsilon$ tends to 0. More precisely, we will prove
 that there exist realizations $\{x_\varepsilon^\theta(t),\,t\geq0\}$ of the above process on the same probability space of a Brownian motion process $\{x(t),\,t\geq0\}$ such that
\begin{equation}
 \lim_{\varepsilon\rightarrow\infty} \max_{0\leq t\leq1} |x_\varepsilon^\theta(t)-x(t)|=0 \quad a.s. \label{fort1}
\end{equation}
 The convergence
of the imaginary part of $z_\varepsilon$ to another
standard Brownian motion $\{y_t;\,t\in[0,1]\}$,  independent of
$\{x_t,t\geq0\}$, follows the same proof. We will follow the method used in \cite{art G-H-RM} to prove the
strong convergence  to a standard
Brownian motion.
We will divide the proof in five steps.

\bigskip

{\it Step 1: Definitions of the processes.} Let $(\Omega,\cal{F},\cal{P})$ be the probability space for a standard Brownian motion $\{x_t,t\geq0\}$ with $x(0)=0$ and let us define:

\begin{enumerate}

\item   for each $\varepsilon>0$,  $\{\epsilon_m^\varepsilon\}_{m \ge 1}$  a sequence of independent identically distributed random variables with law exponential of parameter $\frac{2}{\varepsilon}$, independent of the Brownian motion $x$,

\item $\{ \eta_m \}_{m \ge 1}$ a sequence of independent identically distributed random variables with law Bernoulli($\frac12$), independent of  $x$ and $\{\epsilon_m^\varepsilon\}_{m \ge 1}$ for all $\varepsilon$.

\item $\{ k_m \}_{m \ge 1}$ a sequence of independent identically distributed random variables such that
$P(k_1=1)=P(k_1=-1)=\frac12$ , independent of  $x$, $\{ \eta_m \}_{m \ge 1}$ and $\{\epsilon_m^\varepsilon\}_{m \ge 1}$ for all $\varepsilon$.

\end{enumerate}

Using these random variables we are able to introduce the following ones:

\begin{enumerate}

\item  $\{b_m \}_{m \ge 0}$ such that $b_0=0$ and
$b_m=\sum_{j=1}^{m}\eta_j$ for $m\geq1$. Clearly $b_m$ has a
Binomial distribution of parameters $(m,\frac12)$ and, for all
$n\in\{0,1,\dots,m\}$,
$P(b_{m+1}=n|b_m=n)=P(b_{m+1}=n+1|b_m=n)=\frac12$.

\item  $\{\xi_m^{\varepsilon,\theta}\}_{m \ge 1}=\{|\cos{(b_{m-1}\theta)}|\epsilon_m^\varepsilon\}_{m\ge 1}$ . This family of random variables is clearly independent of $x$.

\end{enumerate}

  Let $\mathscr{B}$ be the $\sigma$-algebra generated by $\{b_m\}_{m\ge 1}$. The sequence of random variables $\{k_m\xi_m^{\varepsilon,\theta}\}_{m \ge 0}$ satisfies
    \begin{eqnarray*}
      &&\E(k_m\xi_m^{\varepsilon,\theta}|\mathscr{B})=0,\\
      &&Var(k_m\xi_m^{\varepsilon,\theta}|\mathscr{B})=\E[(\xi_m^{\varepsilon,\theta})^2|\mathscr{B}]=\frac{\varepsilon^2}{2}[\cos(b_{m-1}\theta)]^2.
    \end{eqnarray*}
By Skorokhod's theorem (\cite{sko}  page 163 or \textit{Lemma 2} in \cite{art G-G2}) for each $\varepsilon>0$ there exists a sequence $\sigma_1^{\varepsilon,\theta},\sigma_2^{\varepsilon,\theta},...$ of nonnegative random variables on $(\Omega,\cal{F},\cal{P})$ so that the sequence $x(\sigma_1^{\varepsilon,\theta}), x(\sigma_1^{\varepsilon,\theta}+\sigma_2^{\varepsilon,\theta}),...,$ has the same distribution as $k_1\xi_1^{\varepsilon,\theta},k_1\xi_1^{\varepsilon,\theta}+k_2\xi_2^{\varepsilon,\theta},...,$ and, for each $m$,
    $$ \E(\sigma_m^{\varepsilon,\theta}|\mathscr{B})=Var(k_m\xi_m^{\varepsilon,\theta}|\mathscr{B})=\frac{\varepsilon^2}{2}[\cos(b_{m-1}\theta)]^2.$$
For each $\varepsilon$  we define $\gamma_0^{\varepsilon,\theta}\equiv0$ and for each $m$
    $$ \gamma_m^{\varepsilon,\theta}=|\beta_m^{\varepsilon,\theta}|^{-1} \left| x\left(\sum_{j=0}^m\sigma_j^{\varepsilon,\theta}\right)-x\left(\sum_{j=0}^{m-1}\sigma_j^{\varepsilon,\theta}\right) \right|, $$
%  where $\beta_i^{\varepsilon,\theta}$ are non-zero random variables defined as follow:
where $\sigma_0^{\varepsilon,\theta}\equiv0$ and $$ \beta_m^{\varepsilon,\theta}=\frac{2}{\varepsilon} \cos(b_{m-1}\theta). $$
Then, the random variables $\gamma_1^{\varepsilon,\theta},\gamma_2^{\varepsilon,\theta},...,$ are independent with common exponential distribution with parameter $\frac{4}{\varepsilon^2}$.
 Indeed
  for any $0\leq n<m$, $x, y \in \R$
    \begin{eqnarray*}
      &&P(\gamma_{n+1}^{\varepsilon,\theta}\leq x,\gamma_{m+1}^{\varepsilon,\theta}\leq y)\\% &=& %\sum_{j=0}^n\sum_{k=j}^{j+(m-n)} P(\gamma_{n+1}^{\varepsilon,\theta}\leq x,\,\gamma_{m+1}^{\varepsilon,\theta}\leq %y,\,b_n=j,\,b_m=k) \\
      &=& \sum_{j=0}^n\sum_{k=j}^{j+(m-n)} P(\gamma_{n+1}^{\varepsilon,\theta}\leq x,\,\gamma_{m+1}^{\varepsilon,\theta}\leq y\,|\,b_n=j,b_m=k) P(b_n=j,b_m=k) \\
      &=& (1-e^{-\frac{4x}{\varepsilon^2}})(1-e^{-\frac{4y}{\varepsilon^2}}) \,\sum_{j=0}^n\sum_{k=j}^{j+(m-n)} P(b_n=j,b_m=k) \\
      &=& P(\gamma_{n+1}^{\varepsilon,\theta}\leq x)  P(\gamma_{m+1}^{\varepsilon,\theta}\leq
      y),
    \end{eqnarray*}
 using that when $b_n$ and $b_m$ are known, $\gamma_{n+1}^{\varepsilon,\theta}$ and $\gamma_{m+1}^{\varepsilon,\theta}$ are independent.

\bigskip

Now, we define $x_\varepsilon^\theta(t), t \ge 0$  to be piecewise linear satisfying
\begin{equation} x_\varepsilon^\theta\left(\sum_{j=1}^{m}\gamma_j^{\varepsilon,\theta}\right)=x\left(\sum_{j=1}^{m}\sigma_j^{\varepsilon,\theta}\right), \qquad  m \ge 1 \label{defreal}
\end{equation}
  and $x_\varepsilon^\theta(0)\equiv0$. Observe that the process $x_\varepsilon^\theta$ has slope $\pm|\beta_m^{\varepsilon,\theta}|$ in the interval $[\sum_{j=1}^{m-1}\gamma_j^{\varepsilon,\theta},\sum_{j=1}^{m}\gamma_j^{\varepsilon,\theta}]$.

\bigskip

  Let $\rho_m^{\varepsilon,\theta}$ be the time of the $m$th change of the absolute values of $\beta_j^{\varepsilon,\theta}$'s, i.e. the time when  $\beta_j^{\varepsilon,\theta}=\frac{2}{\varepsilon}\cos[(m-1)\theta]$ and $\beta_{j+1}^{\varepsilon,\theta}=\frac{2}{\varepsilon}\cos(m\theta)$, that is  when the slope of $x_\varepsilon^\theta(\cdot)$ changes from $\pm\frac{2}{\varepsilon}|\cos[(m-1)\theta]|$ to $\pm\frac{2}{\varepsilon}|\cos(m\theta)|$. Then the increments $\rho_m^{\varepsilon,\theta}-\rho_{m-1}^{\varepsilon,\theta}, m \geq 1$,  with $\rho_0^{\varepsilon,\theta}\equiv0$ are independent and exponentially distributed with a parameter $\frac{2}{\varepsilon^2}$. Indeed, since
   $$P\left(\beta_m^{\varepsilon,\theta}=\frac{2}{\varepsilon}\cos(n\theta)\,\big|\,\beta_{m-1}^{\varepsilon,\theta}=\frac{2}{\varepsilon}\cos[(n-1)\theta]\right)=\frac12$$ for every $n=0,1,\cdots,m$, we can write $\rho_1^{\varepsilon,\theta}=\gamma_1^{\varepsilon,\theta}+\dots+\gamma_{\widehat{N}}^{\varepsilon,\theta}$, where $P(\widehat{N}=n)=2^{-n}$ for $n\in\N$. Therefore $\rho_1^{\varepsilon,\theta}$ has an exponential distribution with parameter $\frac{2}{\varepsilon^2}$  (see \cite{fel}). Likewise, each increment $\rho_m^{\varepsilon,\theta}-\rho_{m-1}^{\varepsilon,\theta}$ has an exponential distribution with parameter $\frac{2}{\varepsilon^2}$ and the increments are independent since they are sum of disjoint blocks of the $\gamma_m^{\varepsilon,\theta}$'s.

On the other hand, let $\tau_m^\varepsilon$ be the time of the $m$th
change of the sign of the slopes. Following the same arguments  for
the times $\rho_m^{\varepsilon,\theta}$ we get that the increments
$\tau_m^\varepsilon-\tau_{m-1}^\varepsilon$, for each $m$, with
$\tau_0^\varepsilon\equiv0$, are independent and exponentially
distributed with a parameter $\frac{2}{\varepsilon^2}$. Moreover the
increments are sum of disjoint blocks of the
$\gamma_m^{\varepsilon,\theta}$'s.

Thus  $x_\varepsilon^\theta$ is a realization of the process (\ref{processreal}).

\bigskip

{\it Step 2 : Decomposition of the convergence.}
Let us come back to the proof of (\ref{fort1}). Recalling that $\gamma_0^{\varepsilon,\theta}\equiv\sigma_0^{\varepsilon,\theta}\equiv0$, by (\ref{defreal}) and the uniform continuity of Brownian motion on $[0,1]$, we have almost surely
\begin{eqnarray*}
  \lim_{\varepsilon\rightarrow0} \,\,\max_{0\leq t\leq1} \left| x_\varepsilon^\theta(t)-x(t) \right| &=& \lim_{\varepsilon\rightarrow0} \,\,\max_{0\leq m\leq\frac{4}{\varepsilon}} \left| x_\varepsilon^\theta\left(\sum_{j=1}^m\gamma_j^{\varepsilon,\theta}\right)-x\left(\sum_{j=1}^m\gamma_j^{\varepsilon,\theta}\right) \right| \\
  &=& \lim_{\varepsilon\rightarrow0} \,\,\max_{0\leq m\leq\frac{4}{\varepsilon}} \left| x\left(\sum_{j=1}^m\sigma_j^{\varepsilon,\theta}\right)-x\left(\sum_{j=1}^m\gamma_j^{\varepsilon,\theta}\right) \right|,
\end{eqnarray*}
and it reduces the proof to check that,
    \begin{equation}\nonumber
      \lim_{\varepsilon\rightarrow0} \,\max_{1\leq m\leq\frac{4}{\varepsilon^2}} \left| \gamma_1^{\varepsilon,\theta}+\dots+\gamma_m^{\varepsilon,\theta}-m\frac{\varepsilon^2}{4} \right|=0 \quad a.s.,
    \end{equation}
and that
    \begin{equation}\label{limdif} \lim_{\varepsilon\rightarrow0} \,\max_{1\leq m\leq\frac{4}{\varepsilon^2}} \left| \sigma_1^{\varepsilon,\theta}+\dots+\sigma_m^{\varepsilon,\theta}-m\frac{\varepsilon^2}{4} \right|=0 \quad a.s.  \end{equation}

\bigskip
The first limit can be obtained easily   by Borel-Cantelli lemma since  by Kolmogorov's inequality, for each $\alpha>0$, we have
    \begin{eqnarray*}
      P\left(\max_{1\leq m\leq\frac{4}{\varepsilon^2}}\left| \gamma_1^{\varepsilon,\theta}+\dots+\gamma_m^{\varepsilon,\theta}-m\frac{\varepsilon^2}{4} \right|\geq\alpha\right) &\leq& \frac{1}{\alpha^2}\sum_{m=1}^{\frac{4}{\varepsilon^2}} Var(\gamma_m^{\varepsilon,\theta}) \\
     % &=& \frac{1}{\alpha^2} \cdot \frac{4}{\varepsilon^2} \cdot \frac{\varepsilon^4}{16} \\
      &=& \frac{\varepsilon^2}{4\alpha^2}.
    \end{eqnarray*}

In order to deal with (\ref{limdif}) we can use the decompostion
\begin{equation}\label{exp1}
\max_{1\leq m\leq\frac{4}{\varepsilon^2}} \left| \sum_{j=1}^m\sigma_j^{\varepsilon,\theta}-m\frac{\varepsilon^2}{4} \right| \! \leq \! \max_{1\leq m\leq\frac{4}{\varepsilon^2}} \left| \sum_{j=1}^m(\sigma_j^{\varepsilon,\theta}-\alpha_j^{\varepsilon,\theta}) \right| + \max_{1\leq m\leq\frac{4}{\varepsilon^2}} \left| \sum_{j=1}^m\alpha_j^{\varepsilon,\theta}-m\frac{\varepsilon^2}{4} \right|, \end{equation}
  where $\alpha_j^{\varepsilon,\theta}:=\E(\sigma_j^{\varepsilon,\theta}|\mathscr{B})=\frac{\varepsilon^2}{2}[\cos(b_{j-1}\theta)]^2$ and
    \begin{eqnarray}\label{eq10}
      &&\left| \sum_{j=1}^m\alpha_j^{\varepsilon,\theta}-m\frac{\varepsilon^2}{4} \right| = \left| \sum_{j=1}^m\frac{\varepsilon^2}{2}[\cos(b_{j-1}\theta)]^2-m\frac{\varepsilon^2}{4} \right|  \\
      &=& \frac{\varepsilon^2}{4} \left| \,\sum_{j=1}^m[1+\cos(2b_{j-1}\theta)]-m\, \right| \nonumber
      = \frac{\varepsilon^2}{4} \left| \sum_{j=1}^m\cos(2b_{j-1}\theta) \right|,
    \end{eqnarray}
Notice now, that  for a fixed $m$,
    \begin{eqnarray}\label{eq11}
      \sum_{j=1}^m\cos(2b_{j-1}\theta) = \sum_{\begin{subarray}{l} 0\leq k\leq n\\ B_{n,m}\end{subarray}}T_k\cos(2k\theta) + Z_{n+1}\cos[2(n+1)\theta],
    \end{eqnarray}
  where $B_{n,m}:=\{k\in\{0,\dots,n\}\mbox{ s.t. }T_0+\cdots+T_n\leq m,\,T_0+\cdots+T_{n+1}>m\}$, $T_k$ are independent identically distributed random variables with $T_k\sim\mbox{Geom}\big(\frac12\big)$ for each $k=0,1,2,\dots$, that is $P(T_k=j)=2^{-j}$ for $j\geq1$ and $\E(T_k)=Var(T_k)=2$, and $0\leq Z_{n+1}\leq T_{n+1}$.
Hence, putting together (\ref{exp1}), (\ref{eq10}) and (\ref{eq11}), it follows that
    \begin{eqnarray*}
& &\max_{1\leq m\leq\frac{4}{\varepsilon^2}} \left| \sum_{j=1}^m\sigma_j^{\varepsilon,\theta}-m\frac{\varepsilon^2}{4} \right| \leq \max_{1\leq m\leq\frac{4}{\varepsilon^2}} \left| \sum_{j=1}^m(\sigma_j^{\varepsilon,\theta}-\alpha_j^{\varepsilon,\theta}) \right|\  \\ & &+ \max_{1\leq m\leq\frac{4}{\varepsilon^2}} \frac{\varepsilon^2}{4} \left| \sum_{\begin{subarray}{l} 0\leq k\leq n\\ B_{n,m}\end{subarray}}T_k\cos(2k\theta) \right| +\max_{\begin{subarray}{c} 1\leq m\leq\frac{4}{\varepsilon^2}\\ B_{n,m}\end{subarray}} \frac{\varepsilon^2}{4} \Big| Z_{n+1}\cos[2(n+1)\theta] \Big| \\
      & & := L_1^\varepsilon + L_2^\varepsilon+L_3^\varepsilon,
    \end{eqnarray*}
and reduces the proof of  (\ref{limdif}) to check that $\lim_{\varepsilon \to 0}  (L_1^\varepsilon + L_2^\varepsilon+L_3^\varepsilon) = 0 \,\, a.s$.

\bigskip

\goodbreak

{\it Step 3: Study of $L_1^\varepsilon$.}  Let $M_n:=\sum_{j=1}^n(\sigma_j^{\varepsilon,\theta}-\alpha_j^{\varepsilon,\theta})$. Let $\mathcal{B}_n$ denote the $\sigma$-algebra generated by $\{\mathscr{B},\sigma_k^{\varepsilon,\theta};\, k\leq n\}$. Clearly, $M_n$ is $\mathcal{B}_n$-measurable and, since $\E(\sigma_n^{\varepsilon,\theta}-\alpha_n^{\varepsilon,\theta}|\mathscr{B})=0$, $\E(M_n|\mathcal{B}_{n-1})=M_{n-1}$. So $|M_n|$ is a submartingale. By Doob's martingale inequality, for each $\alpha>0$
\begin{equation} P\left( \max_{1\leq m\leq\frac{4}{\varepsilon^2}} |M_m|\geq\alpha \right) \leq \frac{1}{\alpha^2} \E\Big(\big|M_{[\frac{4}{\varepsilon^2}]}\big|^2\Big), \label{dd} \end{equation}
where  $[v]$ is the integer part of $v$.
Fixed $b_j$'s, $\{\sigma_j^{\varepsilon,\theta}\}$'s are independent, and  so
\begin{equation}\label{aa}
      \E[(\sigma_j^{\varepsilon,\theta}-\alpha_j^{\varepsilon,\theta})(\sigma_k^{\varepsilon,\theta}-\alpha_k)]= \E\big[\E\big(\sigma_j^{\varepsilon,\theta}-\alpha_j^{\varepsilon,\theta}\big|\mathscr{B}\big) \E\big(\sigma_k^{\varepsilon,\theta}-\alpha_k\big|\mathscr{B}\big)\big]
      = 0.
    \end{equation}
On the other hand
\begin{equation}\label{bb}\E[\sigma_j^{\varepsilon,\theta}\alpha_j^{\varepsilon,\theta}] = \E[\E(\sigma_j^{\varepsilon,\theta}\alpha_j^{\varepsilon,\theta}|\mathscr{B}] ]= \E[\alpha_j^{\varepsilon,\theta}\E(\sigma_j^{\varepsilon,\theta}|\mathscr{B}]] = \E[(\alpha_j^{\varepsilon,\theta})^2]. \end{equation}
 Using (\ref{aa}) and (\ref{bb}), we get that
   \begin{eqnarray}\label{eq_sigma1}
  & &    \E\Big(\big|M_{[\frac{4}{\varepsilon^2}]}\big|^2\Big) \nonumber  \\
& &=
      \sum_{j=1}^{[\frac{4}{\varepsilon^2}]}\E[(\sigma_j^{\varepsilon,\theta}-\alpha_j^{\varepsilon,\theta})^2]+2\sum_{j=1}^{[\frac{4}{\varepsilon^2}]-1}\sum_{k=j+1}^{[\frac{4}{\varepsilon^2}]}\E[(\sigma_j^{\varepsilon,\theta}-\alpha_j^{\varepsilon,\theta})(\sigma_k^{\varepsilon,\theta}-\alpha_k^{\varepsilon,\theta})]\nonumber \\
&&  = \sum_{j=1}^{[\frac{4}{\varepsilon^2}]}\E[(\sigma_j^{\varepsilon,\theta})^2] -
      \sum_{j=1}^{[\frac{4}{\varepsilon^2}]}\E[(\alpha_j^{\varepsilon,\theta})^2].
     \end{eqnarray}
Recalling that by Skorokhod's theorem (\cite{sko},  page 163) there exists a positive constant $C_1$ such that
    $$ \E[(\sigma_j^{\varepsilon,\theta})^2|\mathscr{B}] \leq C_1\,\E[(\xi_j^{\varepsilon,\theta})^4|\mathscr{B}] = C_1 4!\left(\frac{\varepsilon}{2}\right)^4[\cos(b_{j-1}\theta)]^4, $$
 (\ref{eq_sigma1}) can be bounded by
    \begin{eqnarray*}
&&\E\Big(\big|M_{[\frac{4}{\varepsilon^2}]}\big|^2\Big)
            \leq C_1 4!\left(\frac{\varepsilon}{2}\right)^4\sum_{j=1}^{[\frac{4}{\varepsilon^2}]}\E[\cos(b_{j-1}\theta)^4] - \left(\frac{\varepsilon^2}{2}\right)^2\sum_{j=1}^{[\frac{4}{\varepsilon^2}]}\E[\cos(b_{j-1}\theta)^4] \\
      &&  =C\varepsilon^4 \,\sum_{j=1}^{[\frac{4}{\varepsilon^2}]}\E[\cos(b_{j-1}\theta)^4]
      \leq
      %C\varepsilon^4\cdot\frac{4}{\varepsilon^2} \\
      %&=&
      4C\varepsilon^2,
    \end{eqnarray*}
where $C$ is a positive constant.
So, from (\ref{dd}) we obtain that for any $\alpha$
    $$ P\left( \max_{1\leq m\leq\frac{4}{\varepsilon^2}} |M_m|\geq\alpha \right) \leq \frac{4C\varepsilon^2}{\alpha^2},$$
and by Borel-Cantelli lemma it follows that
    $ \lim_{\varepsilon\rightarrow0} L_1^\varepsilon=0 \, a.s.$

\bigskip

{\it Step 4: Study of $L_2^\varepsilon$.}  Since $n\leq m-1$, we have
    \begin{eqnarray*}
   L_2^\varepsilon &\leq& \max_{0\leq n\leq\frac{4}{\varepsilon^2}-1} \frac{\varepsilon^2}{4} \left| \sum_{k=0}^n T_k\cos(2k\theta) \right| \\
      &\leq& \max_{0\leq n\leq\frac{4}{\varepsilon^2}-1} \frac{\varepsilon^2}{4} \left| \sum_{k=0}^n (T_k-2)\cos(2k\theta) \right| + \max_{0\leq n\leq\frac{4}{\varepsilon^2}-1} \frac{\varepsilon^2}{2} \left| \sum_{k=0}^n \cos(2k\theta) \right|
  \\ & :=&  L_{21}^\varepsilon + L_{22}^\varepsilon.
    \end{eqnarray*}

Let us prove that  $ L_{21}^\varepsilon$ vanishes when $\varepsilon$ goes to 0. Let $\mathcal{F}^n$ denote the $\sigma$-algebra generated by $T_k$ for $k\leq n$. Define $M'_n:=\frac{\varepsilon^2}{4}\sum_{k=0}^n(T_k-2)\cos(2k\theta)$. It is easy to see that $M'_n$ is a martingale. By Doob's martingale inequality, for each $\alpha>0$
  \begin{eqnarray*}
    &&P\left( \max_{0\leq n\leq\frac{4}{\varepsilon^2}-1} |M'_n|>\alpha \right)\\
     &\leq& \frac{1}{\alpha^2} \,\E\left(\big|M'_{[\frac{4}{\varepsilon^2}]-1}\big|^2\right)
    = \frac{1}{\alpha^2} \,\E\left( \frac{\varepsilon^4}{16}\left|\sum_{k=0}^{[\frac{4}{\varepsilon^2}]-1}(T_k-2)\cos(2k\theta)\right|^2\,
    \right)\\
    &=& \frac{1}{\alpha^2}\,\frac{\varepsilon^4}{16} \,\sum_{k=0}^{[\frac{4}{\varepsilon^2}]-1}\E[(T_k-2)^2]\cos^2(2k\theta)
    \leq  \frac{1}{\alpha^2}\,\frac{\varepsilon^4}{16}\,\frac{4}{\varepsilon^2}
    = \frac{\varepsilon^2}{2\alpha^2},
  \end{eqnarray*}
  where we have used that $(T_k-2)$'s are independent and centered.
Therefore, by Borell-Cantelli lemma
   $ \lim_{\varepsilon\rightarrow0}  L_{21}^\varepsilon=0 $  a.s.

 On the other hand, since for any $n$, we get that
    \begin{eqnarray*}
      &&\sum_{k=0}^n\cos(2k\theta) \\&=& \frac12\left( \sum_{k=0}^ne^{i2k\theta}+\sum_{k=0}^ne^{-i2k\theta} \right)
      = \frac12\left( \frac{1-e^{i2(n+1)\theta}}{1-e^{i2\theta}}+\frac{1-e^{-i2(n+1)\theta}}{1-e^{-i2\theta}} \right) \\
      %&=& \frac12 \left(\frac{1-e^{-i2\theta}-e^{i2(n+1)\theta}+e^{i2n\theta}+1-e^{-i2(n+1)\theta}-e^{i2\theta}+e^{-i2n\theta}}{1-e^{i2\theta}-e^{-i2\theta}+1}\right) \\
      %&=& \frac12\left(\frac{2-2\cos(2\theta)-2\cos\big(2(n+1)\theta\big)+2\cos(2n\theta)}{2-2\cos(2\theta)} \right)\\
      %&=&
      %\frac12\left(\frac{1-\cos(2\theta)-\cos\big(2(n+1)\theta\big)+\cos(2n\theta)}{1-\cos(2\theta)}\right)\\
 &=&
 \frac12\left(1+\frac{-\cos\big(2(n+1)\theta\big)+\cos(2n\theta)}{1-\cos(2\theta)}\right)
 \leq\frac12\left(1+\frac{2}{1-\cos(2\theta)}\right),
    \end{eqnarray*}
it yields that
   $ \lim_{\varepsilon\rightarrow0}  L_{22}^\varepsilon=0 $, a.s.

\bigskip

{\it Step 5: Study of $L_3^\varepsilon$.}
 Since $n\leq m-1$ and $Z_{n+1}\leq T_{n+1}$,
  \begin{equation}\label{eq13}
 L_3^\varepsilon\leq \max_{\begin{subarray}{c} 1\leq m\leq\frac{4}{\varepsilon^2}\\ B_{n,m}\end{subarray}} \frac{\varepsilon^2}{4} T_{n+1} \leq \max_{0\leq n\leq\frac{4}{\varepsilon^2}-1} \frac{\varepsilon^2}{4} T_{n+1}.
  \end{equation}
  Thus, it is sufficient to see that
    $$ \lim_{n\rightarrow\infty}\,\max_{1\leq k\leq n}\,\frac{T_k}{n}=0 \qquad \mbox{a.s.} $$
 Fixed $\delta>0$, we have that
    \begin{eqnarray*}
      P\left( \max_{1\leq k\leq n}\frac{T_k}{n}>\delta \right) &=& P\left( \max_{1\leq k\leq n}T_k>\delta n \right) = 1-\Big(P(T_k\leq\delta n)\Big)^n \\
      &=& 1-\left( \sum_{j=1}^{[\delta n]}\frac{1}{2^j} \right)^n = 1-\left( 1-\frac{1}{2^{[\delta n]}}
      \right)^n.
    \end{eqnarray*}
 Since
    $ \lim_{n\rightarrow\infty} 1-\left( 1-\frac{1}{2^{[\delta n]}} \right)^n = 0 $ it follows that
    $$ \max_{1\leq k\leq n}\,\frac{T_k}{n} \xrightarrow[n\rightarrow\infty]{P} 0. $$
 Finally the almost sure convergence follows   for the fact that
    $$ \sum_{n=1}^\infty \left[ 1-\left( 1-\frac{1}{2^{[\delta n]}} \right)^n \right]<\infty,$$
proved in Lemma \ref{serie}   in the Appendix.

\hfill $\square$

\section{Proof of rate of convergence}\label{rates}

  In this section we will prove the rate of convergence of the processes $z_\varepsilon^\theta(t)$. Although the proof follows the structure of part b) of Theorem 1 in \cite{art G-G}, some of the terms appearing have to be computed in a new way.

\bigskip

{\it Proof of Theorem \ref{thm_rate}.}
       To prove the theorem it is sufficient to check that, for any $q>0$,
    $$ P\left( \max_{0\leq t\leq1} |x_\varepsilon^\theta(t)-x(t)| > \alpha\,\varepsilon^{\frac12}\left(\log{\frac1\varepsilon}\right)^\frac52 \right) = o(\varepsilon^q)  \qquad \mbox{as} \quad \varepsilon\rightarrow0 $$
  and
    $$ P\left( \max_{0\leq t\leq1} |y_\varepsilon^\theta(t)-y(t)| > \alpha'\,\varepsilon^{\frac12}\left(\log{\frac1\varepsilon}\right)^\frac52 \right) = o(\varepsilon^q)  \qquad \mbox{as} \quad \varepsilon\rightarrow0 $$
  where $\alpha$ and $\alpha'$ are two positive constants depending on $q$.
We will analyze the rate of convergence for the real part. The results for the imaginary part can be  obtained by similar computations.

  Recall  that $\gamma_0^{\varepsilon,\theta}\equiv\sigma_0^{\varepsilon,\theta}\equiv0$ and define
    $$ \Gamma_m^{\varepsilon,\theta}=\sum_{j=0}^m \gamma_j^{\varepsilon,\theta} \qquad \mbox{and} \qquad \Lambda_m^{\varepsilon,\theta}=\sum_{j=0}^m \sigma_j^{\varepsilon,\theta}. $$
Set
    $$ J^\varepsilon \equiv \max_{0\leq m\leq\frac{4}{\varepsilon^2}} \,\max_{0\leq r\leq\gamma_{m+1}^{\varepsilon,\theta}} \big|x_\varepsilon^\theta(\Gamma_m^{\varepsilon,\theta}+r)-x(\Gamma_m^{\varepsilon,\theta}+r)\big|. $$
Since $ x_\varepsilon^\theta$ is piecewise linear and using  the definition of $\gamma_m^{\varepsilon,\theta}$, notice that
    \begin{eqnarray*}
      x_\varepsilon^\theta(\Gamma_m^{\varepsilon,\theta}+r) %&=&% x_\varepsilon^\theta(\Gamma_m^{\varepsilon,%\theta})+\frac{x_\varepsilon^\theta(\Gamma_{m+1}^{\varepsilon,\theta})-x_\varepsilon^\theta(\Gamma_m^{\varepsilon,%\theta})}{\Gamma_{m+1}^{\varepsilon,\theta}-\Gamma_m^{\varepsilon,\theta}}\,r \\
      &=& x(\Lambda_m^{\varepsilon,\theta})+\frac{x(\Lambda_{m+1}^{\varepsilon,\theta})-x(\Lambda_m^{\varepsilon,\theta})}{\gamma_{m+1}^{\varepsilon,\theta}}\,r \\
      &=& x(\Lambda_m^{\varepsilon,\theta})+|\beta_{m+1}^{\varepsilon,\theta}| \cdot \sgn\Big(x(\Lambda_{m+1}^{\varepsilon,\theta})-x(\Lambda_m^{\varepsilon,\theta})\Big) r.
    \end{eqnarray*}

  Thus,
    \begin{eqnarray*}
      J^\varepsilon  &\leq& \max_{0\leq m\leq\frac{4}{\varepsilon^2}} \big|x(\Lambda_m^{\varepsilon,\theta})-x(\Gamma_m^{\varepsilon,\theta})\big| + \max_{1\leq m\leq\frac{4}{\varepsilon^2}+1} |\beta_m^{\varepsilon,\theta}|\gamma_m^{\varepsilon,\theta} \\ && +\max_{0\leq m\leq\frac{4}{\varepsilon^2}} \,\max_{0\leq r\leq\gamma_{m+1}^{\varepsilon,\theta}} \big|x(\Gamma_m^{\varepsilon,\theta})-x(\Gamma_m^{\varepsilon,\theta}+r)\big| \\
      &\leq& \max_{0\leq m\leq\frac{4}{\varepsilon^2}} \Big|x(\Lambda_m^{\varepsilon,\theta})-x\Big(\frac{m\varepsilon^2}{4}\Big)\Big| + \max_{0\leq m\leq\frac{4}{\varepsilon^2}} \Big|x(\Gamma_m^{\varepsilon,\theta})-x\Big(\frac{m\varepsilon^2}{4}\Big)\Big| \\
      &&  + \max_{0\leq m\leq\frac{4}{\varepsilon^2}} \,\max_{0\leq r\leq\gamma_{m+1}^{\varepsilon,\theta}} \big|x(\Gamma_m^{\varepsilon,\theta})-x(\Gamma_m^{\varepsilon,\theta}+r)\big| + \max_{1\leq m\leq\frac{4}{\varepsilon^2}+1} |\beta_m^{\varepsilon,\theta}|\gamma_m^{\varepsilon,\theta} \\
      &:=& J_1^\varepsilon+J_2^\varepsilon+J_3^\varepsilon+J_4^\varepsilon,
    \end{eqnarray*}
and for any $a_\varepsilon>0$,
$$ P(J^\varepsilon>a_\varepsilon) \leq \sum_{j=1}^4P\Big(J_j^\varepsilon>\frac{a_\varepsilon}{4}\Big) := I_1^\varepsilon+I_2^\varepsilon+I_3^\varepsilon+I_4^\varepsilon.$$

We will study the four terms separately.

\bigskip

{\it 1. Study of the term $I_4^\varepsilon.$}
 Since $\beta_m^{\varepsilon,\theta}=\frac{2}{\varepsilon}\cos(b_{m-1}\theta)$ and
  $\gamma_m^{\varepsilon,\theta}$'s are independent exponentially distributed variables with parameter $\frac{4}{\varepsilon^2}$, this term can be handled as term D in Theorem 1 in \cite{art G-G}. Indeed
    \begin{eqnarray*}
     I_4^\varepsilon &=& P\left( \max_{1\leq m\leq\frac{4}{\varepsilon^2}+1} |\beta_m^{\varepsilon,\theta}|\gamma_m^{\varepsilon,\theta}>\frac{a_\varepsilon}{4} \right)
      \leq P\left( \max_{1\leq m\leq\frac{4}{\varepsilon^2}+1} \gamma_m^{\varepsilon,\theta}>\frac{a_\varepsilon\varepsilon}{8} \right) \\
      &=& 1-P\left( \gamma_m^{\varepsilon,\theta}\leq\frac{a_\varepsilon\varepsilon}{8} \right)^{\frac{4}{\varepsilon^2}+1}
      = 1-\left( 1-e^{-\frac{a_\varepsilon}{2\varepsilon}}
      \right)^{\frac{4}{\varepsilon^2}+1}.
      \end{eqnarray*}
For $a_\varepsilon$ of the type $\alpha\,\varepsilon^\frac12\left(\log\frac{1}{\varepsilon}\right)^\beta$, with $\alpha$ and $\beta$ are positive arbitrary fixed constants, we can see that
    \begin{equation}\label{limD}
      \lim_{\varepsilon\rightarrow0} \frac{1-\left( 1-e^{-\frac{a_\varepsilon}{2\varepsilon}} \right)^{\frac{4}{\varepsilon^2}+1}}{\varepsilon^q}=0 \qquad \mbox{a.s.}
    \end{equation}
and so $I_4^\varepsilon=o(\varepsilon^q)$. Thus, the result is true with
$\beta=\frac52$.

\bigskip

{\it 2. Study of the term $I_1^\varepsilon.$} Let $\delta_\varepsilon>0$. Using the same decomposition that in the previous Section (see (\ref{eq10}) and (\ref{eq11})) we can write
    \begin{eqnarray*}
      I_1^\varepsilon
      &\leq& P\left( \max_{0\leq m\leq\frac{4}{\varepsilon^2}} \, \max_{|s|\leq\delta_\varepsilon} \Big|x\Big(\frac{m\varepsilon^2}{4}+s\Big)-x\Big(\frac{m\varepsilon^2}{4}\Big)\Big|>\frac{a_\varepsilon}{4} \right) \\
      &&+ P\left( \max_{1\leq m\leq\frac{4}{\varepsilon^2}} \Big|\Lambda_m^{\varepsilon,\theta}-\frac{m\varepsilon^2}{4}\Big|>\delta_\varepsilon \right)\\
%      &\leq& P\left( \max_{0\leq m\leq\frac{4}{\varepsilon^2}} \, \max_{|s|\leq\delta_\varepsilon} \Big|%x\Big(\frac{m\varepsilon^2}{4}+s\Big)-x\Big(\frac{m\varepsilon^2}{4}\Big)\Big|>\frac{a_\varepsilon}{4} \right) \\
%      &&+ P\left( \max_{1\leq m\leq\frac{4}{\varepsilon^2}} \bigg|\sum_{j=1}^m(\sigma_j^{\varepsilon,%\theta}-\alpha_j^{\varepsilon,\theta})\bigg|>\frac{\delta_\varepsilon}{2} \right) + P\left( \max_{1\leq %m\leq\frac{4}%{\varepsilon^2}} \bigg|\sum_{j=1}^m\alpha_j^{\varepsilon,\theta}-\frac{m\varepsilon^2}{4}\bigg|%>\frac{\delta_\varepsilon}{2} \right)\\
 &\leq& P\left( \max_{0\leq m\leq\frac{4}{\varepsilon^2}} \, \max_{|s|\leq\delta_\varepsilon} \Big|x\Big(\frac{m\varepsilon^2}{4}+s\Big)-x\Big(\frac{m\varepsilon^2}{4}\Big)\Big|>\frac{a_\varepsilon}{4} \right) \\
      &&+ P\left( \max_{1\leq m\leq\frac{4}{\varepsilon^2}} \bigg|\sum_{j=1}^m(\sigma_j^{\varepsilon,\theta}-\alpha_j^{\varepsilon,\theta})\bigg|>\frac{\delta_\varepsilon}{2} \right)  \\
      &&+
 P\left( \max_{1\leq m\leq\frac{4}{\varepsilon^2}} \frac{\varepsilon^2}{4} \Bigg|\sum_{\begin{subarray}{l} 0\leq k\leq n\\ B_{n,m}\end{subarray}}T_k\cos(2k\theta)\Bigg|>\frac{\delta_\varepsilon}{4} \right) \\
      &&+ P\left( \,\max_{\begin{subarray}{c} 1\leq m\leq\frac{4}{\varepsilon^2}\\ B_{n,m}\end{subarray}} \frac{\varepsilon^2}{4} \Big| Z_{n+1}\cos[2(n+1)\theta] \Big|>\frac{\delta_\varepsilon}{4} \right)\\
      &=& I_{11}^\varepsilon+ I_{12}^\varepsilon+ I_{13}^\varepsilon++ I_{14}^\varepsilon,
    \end{eqnarray*}
recalling that $\alpha_j^{\varepsilon,\theta}=\E(\sigma_j^{\varepsilon,\theta}|\mathscr{B})=\frac{\varepsilon^2}{2}[\cos(b_{j-1}\theta)]^2$ and where $B_{n,m}:=\{k\in\{0,\dots,n\}\mbox{ s.t. }T_0+\cdots+T_n\leq m,\,T_0+\cdots+T_{n+1}>m\}$, $T_k$ are independent identically distributed random variables with $T_k\sim\mbox{Geom}\big(\frac12\big)$ for each $k=0,1,2,\dots$ and $0\leq Z_{n+1}\leq T_{n+1}$.  We will study again the four terms separately.

\medskip

{\it 2.1. Study of the term $I_{12}^\varepsilon.$}
 In Section \ref{cap_realp} we have seen that $|M_n|=\big|\sum_{j=1}^m(\sigma_j^{\varepsilon,\theta}-\alpha_j^{\varepsilon,\theta})\big|$ is a submartingale, so
    \begin{eqnarray}
      I_{12}^\varepsilon &=& P\left( \max_{1\leq m\leq\frac{4}{\varepsilon^2}} \bigg|\sum_{j=1}^m\Big(\frac{4}{\varepsilon^2}\sigma_j^{\varepsilon,\theta}-2[\cos(b_{j-1}\theta)]^2\Big)\bigg|>\frac{2\delta_\varepsilon}{\varepsilon^2} \right) \nonumber \\
      &\leq& \left(\frac{\varepsilon^2}{2\delta_\varepsilon}\right)^{2p} \, \E\left[\left(\sum_{m=1}^{  [4/\varepsilon^2] }\bigg(\frac{4}{\varepsilon^2}\sigma_m^{\varepsilon,\theta}-2[\cos(b_{m-1}\theta)]^2\bigg)\right)^{2p}\right],
%      &\leq& \left(\frac{2\varepsilon^2}{\delta_\varepsilon}\right)^{2p} \, %\E\left[\left(\sum_{m=1}^{4/\varepsilon^2}\bigg(\frac{4}{\varepsilon^2}\sigma_m^{\varepsilon,%\theta}-2[\cos(b_{m-1}\theta)]^2\bigg)\right)^{2p}\right]
      \label{eq4}
    \end{eqnarray}
  for any $p\geq1$, by Doob's martingale inequality.

Set $Y_m:=\frac{4}{\varepsilon^2}\sigma_m^{\varepsilon,\theta}-2[\cos(b_{m-1}\theta)]^2$. Using H\"older's inequality, we obtain
    \begin{eqnarray}
      \E\left[\Bigg(\sum_{m=1}^{[4/\varepsilon^2]}Y_m\Bigg)^{2p}\right] &=& \sum_{\begin{subarray}{c} |u|=2p\\ u_m\neq1\,\forall m \end{subarray}}{2p\choose u} \E\Big(Y_1^{u_1}\cdots Y_{[4/\varepsilon^2]}^{u_{[4/\varepsilon^2]}}\Big) \label{eq5} \\
      &\leq& \sum_{\begin{subarray}{c} |u|=2p\\ u_m\neq1\,\forall m \end{subarray}}{2p\choose u} \big[\E\big(Y_1^{2p}\big)\big]^{u_1/2p}\cdots\big[\E\big(Y_{[4/\varepsilon^2]}^{2p}\big)\big]^{u_{[4/\varepsilon^2]}/2p}. \nonumber
    \end{eqnarray}
  where $u=(u_1,\dots,u_{[4/\varepsilon^2]})$ with $|u|=u_1+\cdots+u_{[4/\varepsilon^2]}$ and
    $$ {2p\choose u} = \frac{(2p)!}{u_1!\cdots u_{[4/\varepsilon^2]}!}. $$
  Notice that in the first equality we have used that if $u_m=1$ for any $m$, then $\E\big(Y_1^{u_1}\cdots Y_{[4/\varepsilon^2]}^{u_{[4/\varepsilon^2]}}\big)=0$. Inded, assume  that $u_m=1$, then
    \begin{eqnarray*}
      &&\E\big(Y_1^{u_1}\cdots Y_{[4/\varepsilon^2]}^{u_{[4/\varepsilon^2]}}\big)= \E\Big[\E\big(Y_1^{u_1}\cdots Y_{[4/\varepsilon^2]}^{u_{[4/\varepsilon^2]}}\big|\mathscr{B}\big)\Big] \\
      &=& \E\Big[\E\big(Y_1^{u_1}\big|\mathscr{B}\big)\cdots\E\big(Y_{m-1}^{u_{m-1}}\big|\mathscr{B}\big)\E\big(Y_m\big|\mathscr{B}\big)\E\big(Y_{m+1}^{u_{m+1}}\big|\mathscr{B}\big)\cdots\E\big(Y_{[4/\varepsilon^2]}^{u_{[4/\varepsilon^2]}}\big|\mathscr{B}\big)\Big]
    \end{eqnarray*}
that is clearly zero  since we have used that fixed $\{b_j\}$'s, $\{\sigma_j\}$'s are independent and $\E\big(Y_m\big|\mathscr{B}\big)=0$.

On the other hand,  by Skorohod's theorem, we have
    \begin{eqnarray*}
      &&\E\big[(\sigma_m^{\varepsilon,\theta})^{2p}\big] = \E\Big[\E\big[(\sigma_m^{\varepsilon,\theta})^{2p}\big|\mathscr{B}\big]\Big]
      \leq \E\Big[2(2p)!\,\E\big[(k_i\xi_m^{\varepsilon,\theta})^{4p}\big|\mathscr{B}\big]\Big] \nonumber \\
      && \leq  2(2p)!\,\E\bigg[(4p)!\left(\frac{\varepsilon}{2}\right)^{4p}\big(\cos(b_{m-1}\theta)\big)^{4p}\bigg]
      \leq 2(2p)!\,(4p)!\left(\frac{\varepsilon}{2}\right)^{4p}.
    \end{eqnarray*}
So, using the inequality $|a+b|^{2p}\leq2^{2p}(|a|^{2p}+|b|^{2p})$, we obtain
     \begin{eqnarray} \nonumber
\E\Big(Y_m^{2p}\Big) &\leq& 2^{2p}\left[ \bigg(\frac{4}{\varepsilon^2}\bigg)^{2p}\E\big[(\sigma_m^{\varepsilon,\theta})^{2p}\big]+2^{2p}\,\E\big[\big(\cos(b_{m-1}\theta)\big)^{4p}\big] \right] \\
    %  &\leq& 2^{2p}\left[ \bigg(\frac{4}{\varepsilon^2}\bigg)^{2p}\E\big[(\sigma_m^{\varepsilon,\theta})^{2p}\big]+2^{2p}
   %   \right]\nonumber\\
      & \leq & 2^{2p}\left[ \bigg(\frac{4}{\varepsilon^2}\bigg)^{2p}2(2p)!\,(4p)!\left(\frac{\varepsilon}{2}\right)^{4p}+2^{2p} \right] \nonumber \\
      &\leq& 2^{2p}\left[ 2(2p)!\,(4p)!+2^{2p} \right]
      \leq 2^{2p+1}\cdot2(2p)!\,(4p)! \nonumber
     \\ &  = & 4\cdot2^{2p}(2p)!\,(4p)!.\label{eq6}
    \end{eqnarray}
 Finally  Lemma \ref{lemF} yields that, for $\displaystyle p\leq1+\frac{\log2}{\log\big[1+\varepsilon(1-\varepsilon^2/4)^\frac12\big]}$,
    \begin{eqnarray}\label{eq7}
      \sum_{\begin{subarray}{c} |u|=2p\\ u_i\neq1\,\forall i \end{subarray}}{2p\choose u} \leq 2^{2p}(2p)!\left(\frac{4}{\varepsilon^2}\right)^p.
    \end{eqnarray}

  Therefore, for $p$ as above,  putting together (\ref{eq4}), (\ref{eq5}), (\ref{eq6}) and (\ref{eq7}) and  applying Stirling formula, $k!=\sqrt{2\pi}\,k^{k+\frac12}e^{-k}e^{\frac{a}{12k}}$, with $0<a<1$, we obtain
    \begin{eqnarray*}
     I_{12}^\varepsilon &\leq& 4\,(\delta_\varepsilon)^{-2p} \,\varepsilon^{2p} \,2^{4p}\big((2p)!\big)^2\,(4p)! \\  &\leq& 4\,(\delta_\varepsilon)^{-2p} \,\varepsilon^{2p} \,2^{4p} \Big[\sqrt{2\pi}(2p)^{2p+\frac12}e^{-2p}e^{\frac{a}{24p}}\Big]^2 \Big[\sqrt{2\pi}(4p)^{4p+\frac12}e^{-4p}e^{\frac{a}{48p}}\Big] \\
      &=& (\delta_\varepsilon)^{-2p} \,\varepsilon^{2p} \,2^{16p+4} \,(2\pi)^{\frac32} \,e^{-8p+\frac{a}{12p}+\frac{a}{48p}} \,p^{8p+\frac32} \\
      &\leq& K_1^p \,(\delta_\varepsilon)^{-2p} \,\varepsilon^{2p} \,p^{8p+3/2}
    \end{eqnarray*}
  where $K_1$ is a constant.

  Let us impose now $ K_1^p \,(\delta_\varepsilon)^{-2p} \,\varepsilon^{2p} \,p^{8p+3/2}=\varepsilon^{2q}$ and $p=\big[\log{\frac1\varepsilon}\big]$. Observe that $p=\big[\log{\frac1\varepsilon}\big]$ fulfills the condition on $p$ of inequality (\ref{eq7}). We get
    \begin{eqnarray}\label{eq9}
      \delta_\varepsilon = K_2 \,\varepsilon^{1-q/[\log{1/\varepsilon}]}
      \,\left[\log{\frac1\varepsilon}\right]^{4+3/(4[\log{1/\varepsilon}])},
    \end{eqnarray}
  where $K_2=\sqrt{K_1}$ is a constant. Clearly, with this $\delta_\varepsilon$,

 \medskip

{\it 2.2. Study of the term $I_{13}^\varepsilon.$}
 As  in Section \ref{cap_realp}, since $n\leq m-1$, we can write
    \begin{eqnarray*}
      I_{13}^\varepsilon &\leq& P\left( \max_{0\leq n\leq\frac{4}{\varepsilon^2}-1} \frac{\varepsilon^2}{4} \Bigg|\sum_{k=0}^n (T_k-2)\cos(2k\theta)\Bigg|>\frac{\delta_\varepsilon}{8} \right) \\
      &&+\, \1_{\left\{\max_{0\leq n\leq\frac{4}{\varepsilon^2}-1} \frac{\varepsilon^2}{2} \Bigg|\sum_{k=0}^n \cos(2k\theta)\Bigg|>\frac{\delta_\varepsilon}{8} \right\}}:=  I_{131}^\varepsilon+ I_{132}^\varepsilon.
    \end{eqnarray*}

Let us begin studying $I_{132}^\varepsilon$. By  (\ref{eq9}),
    \begin{eqnarray*}
     I_{132}^\varepsilon &\leq& \1_{\left\{ \frac{\varepsilon^2}{2}K>\frac{\delta_\varepsilon}{8} \right\}}= \1_{\left\{ \frac{\varepsilon^2}{2}K>\frac{K_2}{8}\,\varepsilon^{1-q/[\log{1/\varepsilon}]} \,\left[\log{\frac1\varepsilon}\right]^{4+3/(4[\log{1/\varepsilon}])} \right\}} \\
      &=& \1_{\left\{ \varepsilon^{-1-q/[\log{1/\varepsilon}]}
      \,\left[\log{\frac1\varepsilon}\right]^{4+3/(4[\log{1/\varepsilon}])}<\frac{4K}{K_2}\right\}},
    \end{eqnarray*}
  and clearly  $I_{132}^\varepsilon=0$ for small $\varepsilon$.

On the other hand,  since $M'_n:=\frac{\varepsilon^2}{4}\sum_{k=0}^n(T_k-2)\cos(2k\theta)$ is a martingale (see Section \ref{cap_realp}), by Doob's martingale inequality,
    \begin{eqnarray*}
       I_{131}^\varepsilon &=& P\left( \max_{0\leq n\leq\frac{4}{\varepsilon^2}-1} \Bigg|\sum_{k=0}^n (T_k-2)\cos(2k\theta)\Bigg|>\frac{\delta_\varepsilon}{2\varepsilon^2} \right) \\
      &\leq& \bigg(\frac{2\varepsilon^2}{\delta_\varepsilon}\bigg)^{2p} \,\E \left[\Bigg(\sum_{k=0}^{[4/\varepsilon^2]-1}
      (T_k-2)\cos(2k\theta)\Bigg)^{2p}\right].
    \end{eqnarray*}
 Set $U_k:=(T_k-2)\cos(2k\theta)$. $U_k$'s are independent and centered random variables and by  H\"older's inequality, we have
    $$ \E\left[\Bigg(\sum_{k=0}^{[4/\varepsilon^2]-1}U_k\Bigg)^{2p}\right] \leq \sum_{\begin{subarray}{c} |u|=2p\\ u_i\neq1\,\forall i \end{subarray}}{2p\choose u} \big[\E\big(U_0^{2p}\big)\big]^{u_0/2p}\cdots\big[\E\big(U_{[4/\varepsilon^2]-1}^{2p}\big)\big]^{u_{[4/\varepsilon^2]-1}/2p}. $$
Let us recall that $T_k\sim\mbox{Geom}(\frac12)$. It is well-known that  $T_k=[\widetilde{T}_k]+1$ where $\widetilde{T}_k\sim\mbox{Exp}(\log2)$.  Hence
   $$
      \E\big(T_k^{2p}\big) \leq \E\big[(\widetilde{T}_k+1)^{2p}\big]
      \leq 2^{2p} \big[\E\big(\widetilde{T}_k^{2p}\big)+1\big]
      \leq 2^{2p} \left[\frac{(2p)!}{(log2)^{2p}}+1\right]
      \leq 2 \,(2p)! \,(4p)!,
    $$
and it follows that
    \begin{eqnarray*}
      &&\E\Big(U_k^{2p}\Big) \leq 2^{2p} \left[ \E\big[\big(T_k\cos(2k\theta)\big)^{2p}\big]+2^{2p}\,\big[\cos(2k\theta)\big]^{2p} \right] \\
      &\leq& 2^{2p}\big[\cos(2k\theta)\big]^{2p} \left( \E\big(T_k^{2p}\big)+2^{2p} \right)
     \leq4\cdot2^{2p}(2p)!\,(4p)! .
    \end{eqnarray*}
Some  inequalities are very crude, but they are helpful since we get the same bounds that in the study of  $I_{12}^\varepsilon.$
Thus, with $\delta_\varepsilon$ as in (\ref{eq9}) we get that $I_{131}^\varepsilon=o(\varepsilon^q)$.

 \medskip

{\it 2.3. Study of the term $I_{14}^\varepsilon$.}  Since $T_k$'s are independent identically distributed random variables with $T_k\sim\mbox{Geom}\big(\frac12\big)$ for each $k$, by (\ref{eq13}),
    \begin{eqnarray*}
     I_{14}^\varepsilon &\leq& P\left( \max_{0\leq n\leq\frac{4}{\varepsilon^2}-1} \,\frac{\varepsilon^2}{4} \,T_{n+1}>\frac{\delta_\varepsilon}{4} \right)
      = 1-P\left( \max_{1\leq n\leq\frac{4}{\varepsilon^2}} T_n \leq \frac{\delta_\varepsilon}{\varepsilon^2} \right) \\
      &=& 1-P\left[\Big( T_n \leq \frac{\delta_\varepsilon}{\varepsilon^2} \Big)\right]^{4/\varepsilon^2}
      \!\!\!\!\!\!\!\!= 1-\Bigg( \sum_{k=1}^{[\delta_\varepsilon/\varepsilon^2]} \frac{1}{2^k} \Bigg)^{4/\varepsilon^2}
      \!\!\!\!\!\!\!\!= 1-\Bigg( 1-\bigg(\frac12\bigg)^{[\delta_\varepsilon/\varepsilon^2]} \Bigg)^{4/\varepsilon^2} \\
      &=& 1-\Big( 1-e^{-\log2\,[\delta_\varepsilon/\varepsilon^2]} \Big)^{4/\varepsilon^2}
      \leq 1-\Big( 1-e^{-\frac{a_\varepsilon}{2\varepsilon}} \Big)^{4/\varepsilon^2 + 1}.
    \end{eqnarray*}
Notice that in the last inequality we have used that $\frac{a_\varepsilon}{2 \varepsilon} \leq \log 2 [\frac{\delta_\varepsilon}{\varepsilon^2}]$.
  The bound for $I_{14}^\varepsilon$ is the same that we got in the study of $I_{4}^\varepsilon$. So, by the same arguments, we  can conclude that $I_{14}^\varepsilon=o(\varepsilon^q)$.
\medskip

{\it 2.4. Study of the term $I_{11}^\varepsilon.$} As in  Theorem 1 in \cite{art G-G}, for  small $\varepsilon$ and using a Doob's martingale inequality for Brownian motion we get
    \begin{eqnarray*}
      I_{11}^\varepsilon &\leq& \sum_{m=0}^{[4/\varepsilon^2]} P\left(\max_{|s|\leq\delta_\varepsilon} \Big|x\Big(\frac{m\varepsilon^2}{4}+s\Big)-x\Big(\frac{m\varepsilon^2}{4}\Big)\Big|>\frac{a_\varepsilon}{4}\right)\nonumber \\
      &=& \Big(\frac{4}{\varepsilon^2}+1\Big)  P\left(\max_{|s|\leq\delta_\varepsilon} \big|x(s)\big|>\frac{a_\varepsilon}{4}\right)
      \leq \frac{32}{\varepsilon^2}  P\left(\max_{0\leq s\leq\delta_\varepsilon} x(s)>\frac{a_\varepsilon}{4}\right) \nonumber\\
      &\leq& \frac{32}{\varepsilon^2}  \exp\left(-\Big(\frac{a_\varepsilon}{4}\Big)^2\frac{1}{2\delta_\varepsilon}\right)
      = \frac{32}{\varepsilon^2} \,\exp\left(-\frac{(a_\varepsilon)^2}{32\delta_\varepsilon}\right).%\label{i11}
    \end{eqnarray*}

Condition
$(32/\varepsilon^2)\exp\big(-(a_\varepsilon)^2/(32\delta_\varepsilon)\big) \leq 32\,\varepsilon^{2q}$ yields that
$$
      a_\varepsilon \ge  K_3 \,\varepsilon^{1/2} \,\varepsilon^{-q/2[\log{1/\varepsilon}]} \,\bigg[\log{\frac1\varepsilon}\bigg]^{2+3/8[\log{1/\varepsilon}]}
      \,\bigg(\log{\frac1\varepsilon}\bigg)^{1/2},
$$
  where $K_3$ is a constant depending on $q$. Notice that $$ a_\varepsilon = \alpha \,\varepsilon^{1/2} \,\bigg(\log{\frac1\varepsilon}\bigg)^{5/2}, $$
  for small $\varepsilon$, where $\alpha$ is a constant that depends on $q$, satisfies such a condition. Thus,
with $\delta_\varepsilon$ as in (\ref{eq9}), it follows that $ I_{11}^\varepsilon=o(\varepsilon^q)$.

\bigskip

{\it 3. Study of the term $I_{2}^\varepsilon$.}
For our $\delta_\varepsilon>0$, we have
    \begin{eqnarray*}
    I_{2}^\varepsilon
      &\leq& P\left( \max_{0\leq m\leq\frac{4}{\varepsilon^2}} \, \max_{|s|\leq\delta_\varepsilon} \Big|x\Big(\frac{m\varepsilon^2}{4}+s\Big)-x\Big(\frac{m\varepsilon^2}{4}\Big)\Big|>\frac{a_\varepsilon}{4} \right) \\
      && + P\left( \max_{0\leq m\leq\frac{4}{\varepsilon^2}} \Big|\Gamma_m^{\varepsilon,\theta}-\frac{m\varepsilon^2}{4}\Big|>\delta_\varepsilon \right)
      = I_{21}^\varepsilon + I_{22}^\varepsilon.
    \end{eqnarray*}

  On one hand, observe that $ I_{21}^\varepsilon = I_{11}^\varepsilon$, thus $ I_{21}^\varepsilon=o(\varepsilon^q)$.

On the other hand, it is easy to check that $\sum_{j=0}^m(\gamma_j^{\varepsilon,\theta}-\frac{\varepsilon^2}{4})$ is a martingale. So
applying  Doob's martingale inequality
    \begin{eqnarray*}
       I_{22}^\varepsilon
      &=& P\left( \max_{0\leq m\leq\frac{4}{\varepsilon^2}} \Bigg|\sum_{j=0}^m\bigg(\frac{4}{\varepsilon^2}\gamma_j^{\varepsilon,\theta}-1\bigg)\Bigg|>\frac{4\delta_\varepsilon}{\varepsilon^2} \right) \\
      &\leq& \left(\frac{\varepsilon^2}{4\delta_\varepsilon}\right)^{2p} \,
      \E\left[\left(\sum_{m=1}^{4/\varepsilon^2}\bigg(\frac{4}{\varepsilon^2}\gamma_m^{\varepsilon,\theta}-1\bigg)\right)^{2p}\right].
    \end{eqnarray*}
Set $V_m:=\frac{4}{\varepsilon^2}\gamma_m^{\varepsilon,\theta}-1$. Notice that $V_m$'s are independent and centered random variables with
    \begin{eqnarray*}
      \E\Big(V_m^{2p}\Big) &\leq& 2^{2p}\left( \bigg(\frac{4}{\varepsilon^2}\bigg)^{2p}\E\big[(\gamma_m^{\varepsilon,\theta})^{2p}\big]+1 \right) \\
      &\leq& 2^{2p}\big((2p)!+1\big)
      \leq 2^{2p+1}(2p)!
      \leq 4\cdot2^{2p}(2p)!\,(4p)!.
    \end{eqnarray*}
Then using an inequality of the type of (\ref{eq5}) and following the same arguments that in the study
of $I_{12}^\varepsilon$, we get that $ I_{22}^\varepsilon=o(\varepsilon^q)$.
\bigskip

{\it 4. Study of the term $I_{3}^\varepsilon$.}
 For $\delta_\varepsilon>0$  defined in (\ref{eq9}) and $a_\varepsilon$  of the type $\alpha\,\varepsilon^\frac12\left(\log\frac{1}{\varepsilon}\right)^\frac52$
    \begin{eqnarray*}
     I_{3}^\varepsilon&\leq& P\left( \max_{0\leq m\leq\frac{4}{\varepsilon^2}} \,\max_{|r|\leq\delta_\varepsilon} \big|x(\Gamma_m^{\varepsilon,\theta})-x(\Gamma_m^{\varepsilon,\theta}+r)\big|>\frac{a_\varepsilon}{4} \right) \\
      && + P\left( \max_{1\leq m\leq\frac{4}{\varepsilon^2}+1} \gamma_m^{\varepsilon,\theta}>\delta_\varepsilon \right)
      := I_{31}^\varepsilon + I_{32}^\varepsilon.
    \end{eqnarray*}

 On one hand, $ I_{31}^\varepsilon=o(\varepsilon^q)$ is proved in the same way as  $ I_{11}^\varepsilon$.

 On the other hand,
    \begin{eqnarray*}
       I_{32}^\varepsilon&=& 1-\left[ P\big(\gamma_m^{\varepsilon,\theta}\leq\delta_\varepsilon\big) \right]^{(4/\varepsilon^2)+1}
      = 1-\Big( 1-e^{-4\delta_\varepsilon/\varepsilon^2}
      \Big)^{(4/\varepsilon^2)+1}.
    \end{eqnarray*}
  Thus  $ I_{32}^\varepsilon=o(\varepsilon^q)$,  similarly as we have proved for $ I_{4}^\varepsilon$.

\bigskip
We have checked now that all the terms in our decomposition are of order $\varepsilon^q$.  More precisely
we have proved that for $a_\varepsilon=\alpha\varepsilon^{\frac12}\big(\log\frac{1}{\varepsilon}\big)^{\frac52}$ and
  $q>0$,
    $$ E := P\left( \max_{0\leq m\leq\frac{4}{\varepsilon^2}} \,\max_{0\leq r\leq\gamma_{m+1}^{\varepsilon,\theta}} \big|x_\varepsilon^\theta(\Gamma_m^{\varepsilon,\theta}+r)-x(\Gamma_m^{\varepsilon,\theta}+r)\big| > a_\varepsilon \right) = o(\varepsilon^q). $$
Let us fix $0<v<1$. Then
    \begin{eqnarray*}
      && P\left( \max_{0\leq t\leq 1-v} \,\big|x_\varepsilon^\theta(t)-x(t)\big| > a_\varepsilon \right) \\
      & \leq& P\left( \max_{0\leq t\leq 1-v} \,\big|x_\varepsilon^\theta(t)-x(t)\big|>a_\varepsilon, \,\Gamma_{4/\varepsilon^2}^{\varepsilon,\theta}\geq 1-v \right) + P\left( \Gamma_{4/\varepsilon^2}^{\varepsilon,\theta} < 1-v \right) \\
      &:=&  E_1 + E_2.
    \end{eqnarray*}

  On one hand, since $E_1\leq E$, we get that  $E_1=o(\varepsilon)$. On the other hand,
    $$ E_2 \leq P\left( \max_{0\leq m\leq\frac{4}{\varepsilon^2}} \,\bigg|\Gamma_m^{\varepsilon,\theta}-\frac{m\varepsilon^2}{4}\bigg|>v \right) \leq I_{22}^\varepsilon, $$
  for small $\varepsilon$. Thus, $E_2=o(\varepsilon^q)$.

We have proved the rate of convergence results  in the interval $[0,1-v]$, but we can extend the argument for any compact interval. So, the proof of Theorem \ref{thm_rate} is completed.

\hfill$\square$

\section{Appendix}

We begin this appendix recalling two technical lemmas that will be useful in our computations.

\begin{prop}\label{2poisson}
Let $\{M_t,t\geq0\}$ be a Poisson process of parameter 2. Set $\{N_t,t\geq0\}$ and $\{N'_t,t\geq0\}$ two other counter processes
  that, at each jump of $M$, each of them jumps or  does not jump with probability $\frac12$, independently of the jumps of the other process and of its past.Then $N$ and $N'$ are Poisson processes of parameter 1 with independent increments on disjoint   intervals.
\end{prop}

\begin{proof} Let us check first that $N$ is a Poisson process of parameter 1. Clearly
  $N_0=0$ and for any $0\leq s<t$ and for $k\in\N\cup\{0\}$,
          \begin{eqnarray*}
            P(N_t-N_s=k) &=& \sum_{n=k}^\infty P(N_t-N_s=k\,|\,M_t-M_s=n)P(M_t-M_s=n) \\
            &=& \sum_{n=k}^\infty {n\choose k} \frac{1}{2^k} \frac{1}{2^{n-k}} \frac{[2(t-s)]^n}{n!}  e^{-2(t-s)} \\
            %&=& e^{-2(t-s)} \sum_{n=k}^\infty {n\choose k} \frac{(t-s)^n}{n!} \\
            %&=& e^{-2(t-s)}\,\frac{(t-s)^k}{k!} \,\sum_{n=k}^\infty \frac{(t-s)^{n-k}}{(n-k)!} \\
            &=& e^{-(t-s)}\,\frac{(t-s)^k}{k!}.
          \end{eqnarray*}
    Finally,   for any $n\geq0$ and for any $0\leq t_1<\cdots<t_{n+1}$, it holds that the increments $N_{t_2}-N_{t_1},\dots,N_{t_{n+1}}-N_{t_n}$ are independent random
        variables. Indeed,  consider $k_i\in\N\cup\{0\}$, using the independence of  the increments of the Poisson process $M$ we
get that
         \begin{eqnarray*}
            &&P(N_{t_2}-N_{t_1}=k_1,\dots,N_{t_{n+1}}-N_{t_n}=k_n \,|\, M_{t_2}-M_{t_1}=m_1, \\
&& \qquad\qquad \dots,M_{t_{n+1}}-M_{t_n}=m_n) \\
            &&\quad =\prod_{i=1}^n \,P(N_{t_{i+1}}-N_{t_i}=k_i \,|\,
            M_{t_{i+1}}-M_{t_i}=m_i).
          \end{eqnarray*}
   Then
      \begin{eqnarray*}
            &&\!\!\!\!\! P(N_{t_2}-N_{t_1}=k_1,\dots,N_{t_{n+1}}-N_{t_n}=k_n) \\
            &=&\!\!\!\!\!\!\!\!\!\!\sum_{\begin{subarray}{c} m_i=0 \\ i=1,\dots,n\end{subarray}}^\infty \!\!\!\!\! P(N_{t_2}-N_{t_1}\!\!=k_1,\dots,N_{t_{n+1}}-N_{t_n}\!\!=k_n \,|\, M_{t_2}-M_{t_1}\!\!=m_1, \\
&& \qquad\qquad\qquad \dots,M_{t_{n+1}}-M_{t_n}\!\!=m_n) \\
            &&\!\!\!\!\!\times \,P(M_{t_2}-M_{t_1}=m_1,\dots,M_{t_{n+1}}-M_{t_n}=m_n) \\
            &=&\!\!\!\!\! P(N_{t_2}-N_{t_1}=k_1) \cdots
            P(N_{t_{n+1}}-N_{t_n}=k_n).
          \end{eqnarray*}

  Using similar arguments we can prove that   for any $k,j\in\N\cup\{0\}$, and for $0\leq s<t<u<v$,
 $$ P(N_t-N_s=k,N'_v-N'_u=j) =P(N_t-N_s=k)P(N'_v-N'_u=j).$$
Thus   $N$ and $N'$ have independent increments on disjoint intervals.
\end{proof}

\begin{lema}\label{serie}
For any $\delta >0$
$$
\sum_{n=1}^\infty \left[ 1-\left( 1-\frac{1}{2^{[\delta n]}} \right)^n \right] <\infty.
$$
\end{lema}

\begin{proof}
We have
    \begin{eqnarray*}
     && \sum_{n=1}^\infty \left[ 1-\left( 1-\frac{1}{2^{[\delta n]}} \right)^n \right]
     \leq \sum_{n=1}^\infty \left[ 1-\left( 1-\frac{1}{2^{\delta n-1}} \right)^n \right]\\
     && \quad = \sum_{n=1}^\infty \frac{2^{(\delta n-1)n}-(2^{\delta n-1}-1)^n}{2^{(\delta n-1)n}} \\
      &&  \quad = \sum_{n=1}^\infty \frac{1}{2^{(\delta n-1)n}} \sum_{k=1}^n {n\choose k}(-1)^{k+1}\,2^{(\delta n-1)(n-k)}\\
&&  \quad
      = \sum_{n=1}^\infty \sum_{k=1}^n {n\choose k}(-1)^{k+1}\,2^{-k(\delta n-1)}\\
 &&  \quad \leq     \sum_{n=1}^\infty \sum_{k=1}^n {n\choose k}2^{-k(\delta n-1)}
      = \sum_{k=1}^\infty 2^{-k(\delta k-1)} \sum_{n=k}^\infty {n\choose k}2^{-\delta k(n-k)}.
 \end{eqnarray*}
Using that $\displaystyle\sum_{n=k}^\infty{n\choose k}x^{n-k}=\frac{1}{(1-x)^{k+1}}$, we can bound the above expression by
$$ \sum_{k=1}^\infty \left( \frac{2^{-(\delta k-1)}}{1-2^{-\delta k}} \right)^k \cdot\frac{1}{1-2^{-\delta k}}
      \leq \frac{1}{1-2^{-\delta}}\cdot\sum_{k=1}^\infty \left( \frac{2^{-(\delta k-1)}}{1-2^{-\delta k}}
      \right)^k.$$
  Finally, we use d'Alembert's ratio test to check the convergence of this series
    $$ \lim_{k\rightarrow\infty} \frac{\left(\frac{2^{-\delta(k+1)+1}}{1-2^{-\delta(k+1)}}\right)^{k+1}}{\left(\frac{2^{-\delta k+1}}{1-2^{-\delta k}}\right)^k} = \lim_{k\rightarrow\infty} \,\frac{2^{1-\delta}\cdot2^{-\delta k}}{1-2^{-\delta}\cdot2^{-\delta k}} \left( 2^{-\delta}\cdot\frac{1-2^{-\delta k}}{1-2^{-\delta(k+1)}} \right)^k = 0 . $$
\end{proof}

Finally, for completeness we also recall a Lemma of  \cite{art G-G} (page 298).

\begin{lema}\label{lemF}
Let $F(k,n)=$ number of ways of putting $k$ balls into $n$ boxes so that no box contains exactly one ball, i.e.,
$$F(k,n)=\sum_{\begin{subarray}{c} \alpha_1+\ldots+\alpha_n=k\\ \alpha_i\neq1\,\forall i \end{subarray}}
\frac{k !}{ \alpha_1! \ldots\alpha_n!}.$$
Then $F(k,n) \le 2^k k! n^{k/2}$ for $k \le 2 +\log4/\log[1+2(1/n-1/n^2)^\frac12].$
\end{lema}

%{\bf Aknowledgements} We would like to thank ...

\end{document}